\documentclass[11pt,twoside,a4paper]{article}

\usepackage[T1]{fontenc}
\usepackage{color}
\usepackage{subfig}
\usepackage{amsmath,amsfonts,amssymb,amsxtra,setspace,xspace,graphicx,lmodern,psfrag,epsfig,color,latexsym}

\usepackage{subfig}
\usepackage{graphicx}

\usepackage{amsthm}

\theoremstyle{definition}
\newtheorem{theorem}{Theorem}
\newtheorem{lemma}{Lemma}
\newtheorem{definition}{Definition}
\newtheorem{Thm}{Theorem}
\newtheorem{proposition}[Thm]{Proposition}

\newtheorem{remark}{Remark}
\usepackage{changes}

\renewcommand{\vec}[1]{{\boldsymbol#1}}

\newcommand{\R}{\mathbb{R}}
\newcommand{\C}{\mathbb{C}}

\newcommand{\ie}{\textit{i.e.}\/, }
\newcommand{\eg}{\textit{e.g.}\/, }

\def\e{\epsilon}

\def\O{\Omega}

\def\Z{\mathbb{Z}}

\def\C{\mathbb{C}}
\def\e{\eta}

\def\ue{u_\e}
\def\uek{u_{\e_k}}

\def\uo{u_0(\vec{x},\vec{y})}

\def\ve{v_\e}

\def\wdto{\buildrel\hbox{${\bf R}$}\over
\rightharpoonup \!\!\!\!\! \rightharpoonup}

\def\dto{\buildrel\hbox{${\bf R}$}\over
\to \!\!\!\!\! \to}

\def\tol2{\buildrel\hbox{$L^2$}\over\longrightarrow}

\def\IOm{\int_{\Omega}}

\def\IOAY{\int\!\!\int_{\Omega\times Y^m}}

\def\limf{\displaystyle{\liminf_{\eta\to 0}}}

\def\xxe{\left(\vec{x},\frac{{\bf R}\vec{x}}{\e}\right)}

\def\xy{(\vec{x},\vec{y})}

\providecommand{\keywords}[1]
{
  \small	
  \textbf{\textit{Keywords---}} #1
}

\begin{document}

\title{Homogenization of quasiperiodic structures and two-scale cut-and-projection convergence}

\author{Niklas Wellander$^{1}$, S\'ebastien Guenneau$^{2}$ and Elena Cherkaev$^{3}$\\
\small $^{1}$Swedish Defence Research Agency, FOI P.O. Box 1165
SE-581 11 Linköping
Sweden \\
\small $^{2}$Aix-Marseille Univ, CNRS, Centrale Marseille, Institut Fresnel, 13013 Marseille, France \\
\small $^{3}$University of Utah, Department of Mathematics, 155 South 1400 East, \\
\small JWB 233 Salt Lake City, UT 84112 USA
}
 
 
\maketitle
\begin{abstract}
Quasiperiodic arrangements of the constitutive materials in composites result in effective properties with very unusual electromagnetic and elastic properties. The paper discusses the cut-and-projection method that is used to characterize effective properties of quasiperiodic materials.
 Characterization of cut-and-projection convergence limits of partial differential operators is presented, and correctors are established.
We provide the proofs of the results announced in  (Wellander et al., 2018) 
and give further examples.
 Applications to problems of interest in physics include electrostatic, elastostatic  and quasistatic magnetic cases.
\end{abstract}

%

\section{Introduction}The search for effective properties of material mixtures dates back to at least the second half of the 19th century starting with the works of Maxwell Garnett, Clausius-Mossotti, and Lord Rayleigh, see \cite{Sihvola1999} for a comprehensive historical survey. The contributions in the field have been in the form of various mixing formulas based on physical insights and simplified models of the effect of dispersed phases in a matrix with landmarks papers by condensed matter physicists such as Bruggeman and Landauer in the first half of the 20th century, see
\cite{levy+stroud1997} for a review, that served as an inspiration for mathematicians such as DeGiorgi and Spagnolo \cite{DeGiorgi+Spagnolo1973},  Tartar\cite{Tartar1978}, Bensoussan, Lions, and Papanicolaou  \cite{Bensoussan+Lions+Papanicolaou1978} and DalMaso  \cite{DalMaso1993}.

A mathematical result in the field of homogenization typically states in which sense the solutions of  sequences of partial differential equations (PDEs) with rapidly varying coefficients converge to the solution of  PDEs with constant coefficients as the coefficient variation becomes more and more rapid. The PDEs with constant coefficients constitute  models of processes taking place in  homogeneous materials, \ie the effective properties of the heterogeneous materials are given by constant coefficients.

There has been a renewed interest in effective medium theories amongst physicists in the last decade as artificial anisotropy plays a crucial role in the design of so-called metamaterials for cloaking of wave propagation and diffusion \cite{FCGE2016}. Such metamaterials are usually aperiodic and high-contrast media.

In \cite{Nguetseng1989}, Nguetseng presented the concept of two-scale convergence, which was further developed in \cite{Allaire1992}. Two-scale convergence turned out to be a very useful concept when homogenizing periodic material mixtures. This is a generalization of the usual weak convergence, in which one uses oscillating test functions to capture oscillations on the same scale as the test functions in the sequence of functions that are investigated.   As a consequence one obtains limit functions that are defined on the product space $\R^n \times ]0,1[^n$. A similar method is the periodic unfolding approach \cite{Cioranescu+etal2008}, in which one first maps the original sequence of functions to a sequence that is defined on $\R^n \times ]0,1[^n$, and then takes the usual weak limit in suitable function spaces, using this extended domain. This approach is similar to the approach proposed in  \cite{Wellander2009}.

It is true that there are composites with a periodic microstructure and many nonperiodic mixtures are very well modeled by periodic composites. However, there are also many material systems that are more cumbersome to model, like the metamaterials mentioned above, such as artificially engineered materials that display a specific pattern, one also encounters stochastic high contrast composites in nature. More intriguing,  researchers also came across large period and quasiperiodic composites in the Koryak Mountains in Eastern Russia: the minerals are made of an alloy of aluminum, copper, and iron \cite{bindi2009}. One can engineer such quasicrystals, for example, mixing two periodic materials may result in a mixture that has a very large periodicity (when there is a common periodicity that is large) or in a quasiperiodic material, in the case of a mixture of materials with rational and irrational periodicity.  Another example is given by moir{\'e} patterns \cite{Sedhed+etal2015}. Back in 1984, a controversial paper by Shechtman, Blech, Gratias, and Cahn on the discovery of a {\it metallic phase with long-range orientational order and no translational symmetry} \cite{Shechtman+etal1984} fueled the interest of physicists and mathematicians alike in quasicrystals.

One feature that makes them particularly appealing to mathematicians is that quasiperiodic materials  can be described by periodic structures in higher spatial dimensions that are cut by hyperplanes and projected onto the lower dimensional space, typically $\R^3$, as proposed by the physicists Duneau and Katz thirty years ago \cite{Duneau+Katz1985}, and then by Whittaker and Whittaker \cite{Whittaker1988}. This opens up the possibility to use standard periodic homogenization tools, e.g., two-scale convergence, to homogenize quasiperiodic materials. To do that, one has to complement existing tools with the cut-and-projection operator, this was done in \cite{Bouchitte+etal2010}.
 In this paper, we revisit this extension, {\em  the two-scale cut-and-projection convergence} method presented in \cite{Bouchitte+etal2010} which  in \cite{Wellander+etal2018}  was further developed to characterize the two-scale limit of differential operators. Interestingly, Braides, Riey, and Solci have independently proposed a $\Gamma$-convergence approach to homogenize Penrose tilings \cite{Braides1986} using general theorems on almost-periodic functions in $W^{1,p}$ spaces \cite{Braides2009}.

The paper is organized as follows. We provide the proofs of the statements in \cite{Wellander+etal2018} and extend the examples to electrostatic, elastostatic, and quasistatic magnetic  problems.
The examples are presented in Section~\ref{sec:Applied_examples}.  Section~\ref{sec:Preliminaries} is devoted to the description of quasicrystals, how they can be modeled, and the basic functional analysis tools to be used in this setting. In Section~\ref{sec:Functionspaces}, we define the differential operators that appear naturally in the context of quasiperiodic materials. We also define and characterize corresponding function spaces to be used in the analysis. Compactness results for the introduced differential operators are presented and proved in Section~\ref{sec:Compactness}. The applied examples discussed in Section~\ref{sec:Applied_examples} are revisited and homogenized in Section~\ref{sec:Examples}.
We present some concluding remarks in Section~\ref{sec:Conclusions}.   Some  non-central lemmas are collected in Appendix~\ref{sec:Appendix}.

\section{Physical problems of interest}\label{sec:Applied_examples}

Throughout this article, we let $\Omega$ be a bounded domain in $\R^3$ with Lipschitz boundary.
\subsection{Electrostatics}
 Consider the electric conductivity problem,
 \begin{equation}  \label{eq:static}
     \left\{
        \begin{aligned}
          -  \mathrm{div} \; \sigma_\e\left( \vec{x} \right)  \mathrm{grad} \; u_\e(\vec{x}) & = f (\vec{x}) \;, \qquad \vec{x} \in \Omega\\
            u_\e|_{\partial\Omega} & = 0
        \end{aligned}
    \right.
 \end{equation}
where $f\in W^{-1,2}(\Omega)$
and $\eta$ is a positive parameter which tends to zero when the fine scale structure in the composite becomes finer and finer.
We assume that $\sigma_\e$ is bounded and coercive, \ie  $\sigma_\e   \in L^{\infty}\left(\Omega; \R^{3\times 3}\right)$ and there exists  a constant $c>0$ such that
   \begin{equation} \label{eq:coersive}
           \sigma_\e(\vec{x}) \xi  \cdot  \xi  \geq c |\xi|^2 , \qquad \mbox{ a.e. }  \vec{x} \in \Omega, \; \forall \xi \in \R^3
   \end{equation}
Standard estimates yield  solutions that are uniformly bounded in $W^{1,2}_0(\Omega)$ with respect to $\e$.

\subsection{Elastostatics}
 Consider the elasticity problem,
 \begin{equation}
 \label{eq:elastatic_eps}
    \left\{
        \begin{aligned}
          -  \mathrm{div} \; {\bf C}_\e\left( \vec{x} \right):  \mathrm{grad} \; {\bf u}_\e(\vec{x}) & = {\bf f} (\vec{x}) \;, \qquad  \vec{x} \in \Omega\\
            {\bf u}_\e{|_{\partial\Omega}} & = {\bf 0}
      \end{aligned}
    \right.
 \end{equation}
 where
${\bf f} \in W^{-1,2}(\Omega,\R^3)$ and $ \mathrm{grad}$ is the symmetrized gradient, \ie
$  \mathrm{grad} \; {\bf u} = \frac{1}{2}(\partial_j u_i +\partial_i u_j)$, $i,j=1, 2, 3$.

We assume that the rank-4 tensor ${\bf C}_\e$, is symmetric bounded and coercive, \ie  ${\bf C}_\e \in L^{\infty}\left(\Omega; \R^{9\times 9}\right)$ and there exists  a constant $c>0$ such that
   \begin{equation} \label{eq:coersive_2}
           C_{\e_{ijkl}}(\vec{x}) \xi_{ij}\xi_{kl}  \geq c \xi_{ij}^2 , \qquad \mbox{ a.e. }  \vec{x} \in \Omega, \; \forall \xi \in \R^{3\times 3}, \; i,j,k,l \in \{1,2,3\}
   \end{equation}
We use Einstein's summation convention over repeated indices. Standard estimates yield  solutions that are uniformly bounded in $W^{1,2}_0(\Omega,\R^3)$ with respect to $\e$.

\subsection{A quasistatic magnetic problem}
Consider the  quasiperiodic heterogeneous quasistatic magnetic problem,
 \begin{equation}  \label{eq:magnetostatic_eps_eta}
     \left\{
        \begin{aligned}
          \mathrm{curl} \, \epsilon^{-1}_\e\left( \vec{x} \right)  \mathrm{curl} \, \vec{u}_\e(\vec{x}) & = \vec{f} (\vec{x}) \,, \qquad \vec{x} \in \Omega\\
            \hat{\vec{\nu}}\times\epsilon^{-1}\mathrm{curl}\,\vec{u}_\e|_{\partial\Omega} & = \vec{0}
        \end{aligned}
    \right.
 \end{equation}
where $ \hat{\vec{\nu}}$ is the unit normal to the boundary. We assume that $\epsilon^{-1}_\e$  is coercive and bounded, \ie  $\epsilon^{-1}_\e   \in L^{\infty}\left(\Omega; \R^{3\times 3}\right)$ and there exists  a constant $c>0$ such that
   \begin{equation} \label{eq:coersive_3}
           \epsilon^{-1}_\e(\vec{x}) \xi  \cdot  \xi  \geq c |\xi|^2 , \qquad \mbox{ a.e. }  \vec{x} \in \Omega, \; \forall \xi \in \R^3
   \end{equation}
The forcing term $\vec{f}$ belongs to the dual of ${\cal H}_0(\mathrm{curl},\Omega)$ and the solutions are uniformly bounded in ${\cal H}_0(\mathrm{curl},\Omega)$ with respect to $\e$.

The homogenization of these three examples when the material properties are assumed to be quasiperiodic will be performed in section 6.

\section{Preliminaries}\label{sec:Preliminaries}

In this section we present the definition of quasicrystals and give the fundamental functional analytical tool, two-scale cut-and-projection convergence.

\subsection{Quasicrystals}
 It is common in homogenization papers to
assume that the medium is periodic, \ie in the electrostatic case that the electric conductivity is a periodic function of the three space
variables. Nonetheless, we shall slightly depart from this
hypothesis by rather assuming that there is some
higher-dimensional
space within which one can define
a periodic conductivity  function
(of more than three variables).
As it turns out, this mathematical game allows for the
analysis of a class of materials which are neither periodic
nor random: Quasicrystalline phases discovered by
Schechtman in the early eighties can be modeled
by taking the cut-and-projection of a periodic structure
in an higher dimensional space (typically $\R^6$ or $\R^{12}$)
onto a hyperplane (such as the Euclidean space $\R^3$).
In the sequel, we will only require the knowledge of a matrix
${\bf R}$ (${\bf R}: \R^n \to \R^m$, $m > n$)
defining this cut-and-projection. In practice, physicists have access to the
opto-geometric properties of a quasicrystal through analysis of the
symmetries of X-ray diffraction patterns (the so-called reciprocal
pseudo-array), which are encompassed in the entries of ${\bf R}$. For instance, the conductivity of
the quasicrystal
$\rm{Al}_{63.5}\rm{Fe}_{12.5}\rm{Cu}_{24}$ is given by ${\bf R}: \R^3 \to \R^6$,
\begin{equation*}
\begin{array}{ll}
\sigma \bigl
( {\bf R}\vec{x} \bigr )
= &
\sigma \bigl ( n_\tau(x_1 + \tau x_2), n_\tau(\tau x_1 + x_3), n_\tau(x_2 +
\tau x_3), \\
& n_\tau(-x_1 + \tau x_2), n_\tau(\tau x_1 - x_3),
n_\tau(-x_2 + \tau x_3) \bigr )
\end{array}
\end{equation*}
where $n_{\tau}$ is the normalization constant
$1/\sqrt{2(2+\tau)}$ with the Golden
number  $\tau$
and $\sigma \in L^{\infty}_{\sharp}(Y^6)$, i.e. the conductivity is bounded almost everywhere on the hypercube $Y^6={]0,1[}^6$
and is periodic.

We note that there is an
ambiguity in the definition of the conductivity as we
could have defined it via a cut-and-projection
from a periodic array in $\R^{12}$ onto $\R^3$ \cite{Duneau+Katz1985}
$\sigma({\bf R}' \vec{x})$ where ${\bf R}': \R^3 \to \R^{12}$, \ie
${\bf R}'$ is a matrix with $12$ rows and $3$ columns and
$\sigma\in L^{\infty}_{\sharp}(Y^{12})$.

However, we shall
see in the next section that the homogenized result does not
actually depend upon ${\bf R}$, that in the general case maps
${\bf R}: \R^n \to \R^m$ ($m > n$), provided it fulfills the
criterion
\begin{equation}\label{eq:criterion}
{\bf R}^T \vec{k} \not = {\bf 0} \; , \; \forall \vec{k} \in
\Z^m \setminus \{{\bf 0}\}
\end{equation}
This criterion corresponds to an irrational slope in the one-dimensional case.
For instance, for $n=1$ and $m=2$, ${\bf R}=(1,\tau)^T$ may generate a 1D quasiperiodic arrangement of materials along the line such that its intervals are of thicknesses $A$ and $B$ alternating according to the   Fibonacci word $ABAABABAABAABABAAB\dots$.  The fibonacci word is the limit of a sequence of words  generated by the rule $S_{n}=S_{n-1}S_{n-2}$ with $S_0 = A$ and $S_1 = AB$.
 Geometrically it can be constructed using a checkerboard-like periodic structure in $2D$ with black and white squares of different materials of respective side lengths $A$ and $B$. The red vertical line of slope $\tau$ in the rotated $y_1-y_2$ reference frame, as shown in Figure \ref{figsiam}, will cut the 2D periodic pattern and produce the sequence  of material properties: $\epsilon_A\epsilon_B\epsilon_A\epsilon_A\epsilon_B\epsilon_A\epsilon_B\epsilon_A\epsilon_A\epsilon_B\epsilon_A\epsilon_A\epsilon_B\dots$, provided that $A/B=\tau$  \cite{Rodriguez+etal2008}.
Of course, the corresponding horizontal line (magenta) will also produce the Fibonacci word. Note that the cut can only start at a unique point in the unit cell to produce the Fibonacci word. Other starting points will give other but similar sequences.

\begin{figure}
\centering
\includegraphics[width=0.7\linewidth]{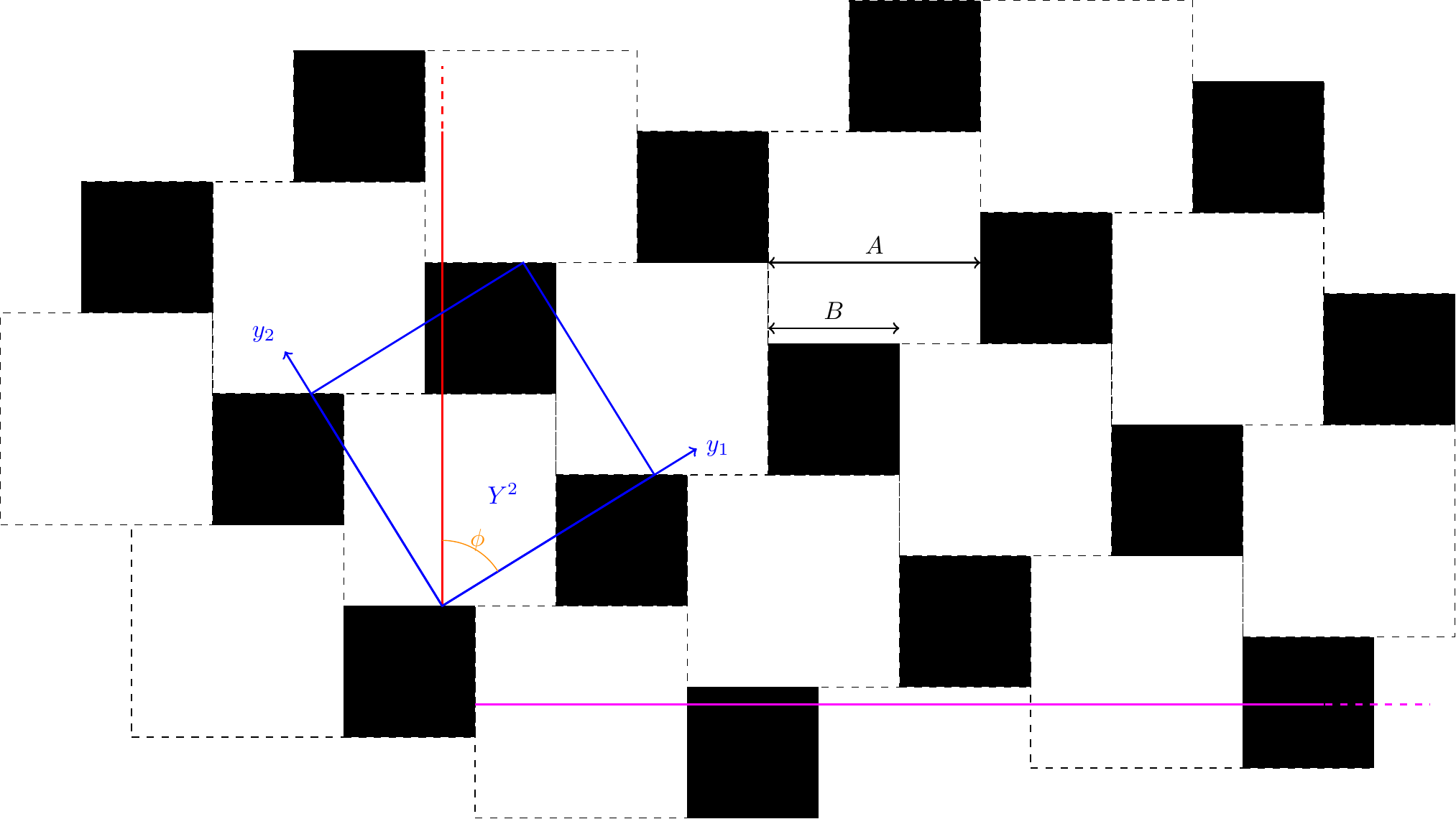}
\caption{Cut-and-projection method applied to a crystal in 2D generating a 1D quasicrystal: The periodic cell $Y^2$, with  side length $B\sqrt{1 + \tau^2}$, makes an irrational angle $\phi$ with the vertical axis (red). When $\phi$ is such that $\tan(\phi)=\tau$, then the vertical axis has slope $\tau$ in the rotated $y_1 - y_2$ coordinate system and it intercepts the crystal which generates an alternation of black and white intervals of lengths $A$ and $B$ such that $A/B=\tau$. In such a way  we get the Fibonacci word defining the sequence of permittivities $\epsilon_A\epsilon_B\epsilon_A\epsilon_A\epsilon_B\epsilon_A\epsilon_B\epsilon_A\epsilon_A\epsilon_B\epsilon_A\epsilon_A\epsilon_B\dots$,.}
\label{figsiam}
\end{figure}
In general, some but not necessarily all the entries of ${\bf R}$ are irrational, as a minimum condition.   We call such projections {\em irrational}, as in \cite{Golden1991}.

\subsection{Two-scale cut-and-projection convergence}

In this section, we recall some properties of two-scale
convergence \cite{Allaire1992} in the quasiperiodic
setting \cite{Bouchitte+etal2010}. Let us
consider a real valued matrix ${\bf R}$  with $m$ rows and $n$
columns ($m>n$). Similarly to the periodic case, our goal is to approximate
an oscillating sequence $\{u_\e(\vec{x})\}_{\eta\in ]0,1[}$ by a
sequence of two-scale locally quasiperiodic   $u_0\left({\vec{x}},\frac{{\bf R}\vec{x}}{\e}\right)$ where $u_0\left({\vec{x}},\cdot\right)$ is $\eta$-periodic on $\R^m$.

As the matrix ${\bf R}$ is not uniquely defined, we first need to
check that if $g$ is a trigonometric polynomial, then the
quasiperiodic function $f = g \circ {\bf R}$ admits the following
(uniquely defined) ergodic mean (for ${\bf R}: \R^n \to \R^m$).
\begin{equation}
L(f) = \lim_{A\to + \infty} \frac{1}{{(2A)}^n}\int_{{]-A,A[}^n} f(\vec{x}) \, \mathrm{d}\vec{x} = \int_{Y^m} g(\vec{y}) \, \mathrm{d}\vec{y} = [g]
\label{ergodic}
\end{equation}
 where $[g]$ denotes the mean of $g$ over the periodic cell $Y^m$ in $\R^m$.
 As shown in \cite{Bouchitte+etal2010} this is the case
 provided that ${\bf R}$ fulfills the criterion  \eqref{eq:criterion}. We recall
the statement and the proof in \cite{Bouchitte+etal2010}  of this elementary
result as it underpins  homogenization of quasicrystals.
\begin{lemma}\label{lem:convergence}
Let ${\bf R}: \R^n\longmapsto \R^m$ ($m\geq n$) satisfy
\eqref{eq:criterion}. Then, (\ref{ergodic}) holds true for any
trigonometric polynomial $g$ on $\R^m$.
\end{lemma}
\begin{proof}
Let $g$ be a trigonometric polynomial defined for every multi-index
$\vec{k}$ in $\Z^m$ by
\begin{equation*}
\forall \vec{y}\in\R^m \; , \; g(\vec{y}) = \sum_{\scriptstyle \mid
\vec{k}\mid \leq k_0} C_{\vec{k}}e^{2\pi i \vec{k} \cdot \vec{y}}
\end{equation*}
where $k_0$ is a positive integer.
Let us now introduce for every $\vec{k}\in\Z^m$ and $\vec{x}\in\R^n$, the quantity
\begin{equation}
\begin{array}{ll}
\displaystyle{L\left(e^{2\pi i \vec{k} \cdot {\bf R}\vec{x}}\right)}
&=\displaystyle{\lim_{A\to + \infty} \frac{1}{{(2A)}^n}\int_{ {]-A,A[}^{n}}e^{2\pi i \vec{k} \cdot {\bf R}\vec{x}} \, \mathrm{d}\vec{x}} \\
&=\displaystyle{\lim_{A\to + \infty} \prod_{j=1}^n \frac{1}{{2A}} \frac{1}{{2i\pi k_l R_{lj}}} {[e^{2\pi i k_l R_{lj}x_j}]}_{-A}^A } =
\displaystyle{\lim_{A\to + \infty} \prod_{j=1}^n {\rm sinc} (2k_l R_{lj} A)}
\end{array}
\label{sine}
\end{equation}
with Einstein's summation convention on repeated index $l$ and
where the cardinal sine function is defined by $\displaystyle{{\rm
sinc} (x) = \frac{\sin (\pi x)}{\pi x}}$ on $\R\setminus\{0\}$, and  $\displaystyle{{\rm
sinc} (0)=1}$.

If ${\bf R}$ satisfies \eqref{eq:criterion}, then $L(e^{2 i
\pi \vec{k} \cdot {\bf R}\vec{x}}) = 0 \; , \; \forall \vec{k} \in
\Z^m\setminus\{{\bf 0}\}$, so that
\begin{equation*}
L(f)=L(g \circ {\bf R}) = \sum_{\scriptstyle \mid
\vec{k}\mid \leq k_0} C_{\vec{k}}L(e^{2\pi i \vec{k} \cdot {\bf R}\vec{x}}) = C_{\bf 0}  = [g]
\end{equation*}
\end{proof}
Lemma~\ref{lem:convergence} is illustrated in Figure~\ref{fig:Torus} in which the two dimensional unit cell  $Y^2$  in Figure~\ref{figsiam} is represented by the unit torus. The torus is cut densely as the line of irrational slope becomes longer and longer.
\begin{figure}[htp]
\centering
\includegraphics[width=0.9\linewidth]{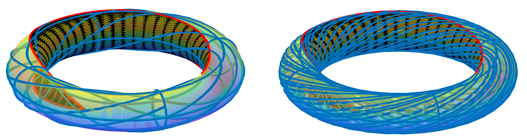}
  \caption{An example of how the unit cell $Y^2$ in Figure~\ref{figsiam} is cut densely by the irrational cut. (Left) The cut of five unit cells, corresponding to five laps on the unit torus. (Right) The cut of twenty five unit  cells.}
  \label{fig:Torus}
\end{figure}
\noindent
The result in Lemma~\ref{lem:convergence} suggests the following definition of two-scale convergence
associated with a matrix ${\bf R}$.
\begin{definition}[Distributional two-scale convergence]
Let $\Omega$ be an open bounded set in $\R^n$ and $Y^m={]0,1[}^m$.
We say that the sequence $(u_\e)$ two-scale converges in the distributional sense towards the
function $u_0\in L^2(\O\times Y^m)$ for a matrix ${\bf R}$  if for every $\varphi\in{\cal
D}(\O;C^\infty_\sharp(Y^m))$
  \begin{equation}\label{eq:strong2scale}
\lim _{\e\to 0} \IOm\ue (\vec{x}) \varphi\left(\vec{x},\frac{ {\bf
R}\vec{x}}{\e}\right) \,\mathrm{d}\vec{x} = \IOAY \uo \phi(\vec{x},\vec{y})\,\mathrm{d}\vec{x}\mathrm{d}\vec{y}
  \end{equation}
\end{definition}

\begin{definition}[Weak two-scale convergence]
Let $\Omega$ be an open bounded set in $\R^n$ and $Y^m={]0,1[}^m$.
We say that the sequence $(u_\e)$ two-scale converges weakly towards the function $u_0\in L^2(\O\times Y^m)$ for a matrix ${\bf R}$  if for every $\varphi\in L^2(\O,C_\sharp(Y^m))$
  \begin{equation}\label{eq:weak2scale}
\lim _{\e\to 0} \IOm\ue (\vec{x}) \varphi\left(\vec{x}, \frac{{\bf R}\vec{x}}{\e}\right) \,\mathrm{d}\vec{x} = \IOAY \uo \phi(\vec{x},\vec{y})\,\mathrm{d}\vec{x}\mathrm{d}\vec{y}
  \end{equation}
\end{definition}
We denote weak two-scale convergence for a matrix ${\bf R}$ with   $\uek \wdto u_0$.
The following result \cite{Bouchitte+etal2010} ensures the existence of such two-scale limits when
the sequence $(u_\e)$ is bounded in $L^2(\O)$ and ${\bf R}$
satisfies \eqref{eq:criterion}.

\begin{proposition}\label{prop:weaktwoscaleprop}
Let $\O$ be an open bounded set in $\R^n$ and $Y^m={]0,1[}^m$. If
${\bf R}: \R^n \to \R^m$ is a linear map   satisfying \eqref{eq:criterion}  and $(\ue)$ is a
bounded sequence in $L^2(\O)$, then there exist a vanishing
subsequence $\e_k$ and a limit $\uo \in L^2(\O\times Y^m)$
($Y^m$-periodic in $\vec{y}$) such that $\uek \wdto u_0$ as $\e_k \to 0$.
\end{proposition}
\noindent
We will need to pass to the limit in integrals
$\displaystyle{\int_{\Omega} u_\e \; v_\e \, \mathrm{d}\vec{x}}$ where
$u_\e\wdto u_0$ and $v_\e\wdto v_0$. For this, we introduce
the notion of strong two-scale
(cut-and-projection) convergence for a matrix ${\bf R}$.
\begin{definition}[Strong two-scale convergence]
A sequence $u_\e$ in $L^2(\O)$ is said to two-scale converge strongly, for a matrix ${\bf R}$, towards a limit $u_0$ in
$L^2_\sharp (\O\times Y^m)$, which we denote $u_\e\dto u_0$, if and
only if $\uek \wdto u_0$ and
\begin{equation*}
\displaystyle{ {\Vert u_\e(\vec{x})\Vert}_{L^2(\O)} \to{\Vert u_0
(\vec{x},\vec{y}) \Vert}_{L^2(\O\times Y^m)}}
\end{equation*}
as $\e_k \to 0$.
\end{definition}

This definition expresses that the effective oscillations of the
sequence $(u_\e)$ have a periodicity that is on the order of $\e$. Moreover,
these oscillations are fully identified by $u_0$.
The following proposition provides us with
a corrector type result for the sequence
$u_\e$ when its limit $u_0$ is smooth enough.
\begin{proposition}
Let ${\bf R}$ be a linear map from $\R^n$ to $R^m$ satisfying
\eqref{eq:criterion}. Let $\ue$ be a sequence bounded in $L^2(\O)$ such
that $\ue\wdto \uo$ (weakly). Then

i) $\ue$ weakly converges in $L^2(\O)$ towards
$\displaystyle{u(\vec{x}) = \int_{Y^m} u_0(\vec{x},\vec{y}) \; \mathrm{d}\vec{y}}$ and
\begin{equation*}
\limf {\Vert\ue\Vert}_{L^2(\O)}\geq {\Vert u_0 \Vert}_{L^2(\O \times
Y^m)} \geq {\Vert u\Vert}_{L^2(\O)}
\end{equation*}

ii) Let $\ve$ be another bounded sequence in $L^2(\O)$ such that
$v_\e\dto v_0$ (strongly), then
\begin{equation*}
\ue\ve\to  w(\vec{x}) \; \hbox{ in ${\cal D}'(\O)$ where
$\displaystyle{w(\vec{x}) = \int_{Y^m}\uo v_0(\vec{x},\vec{y})
\,\mathrm{d}\vec{y}}$}
\end{equation*}

iii) If ${\displaystyle\ue\dto u_0(\vec{x},\vec{y})}$ strongly and
$\displaystyle {{\left\Vert u_0\xxe \right\Vert}_{L^2(\O)} \to{\Vert
u_0\xy\Vert}_{L^2(\O\times Y^m)}}$, then
\begin{equation}
\displaystyle{{\left\| \ue - u_0\xxe \right\|}_{L^2(\O)}\to  0}
\end{equation}
\end{proposition}

Classes of functions such that $\displaystyle {{\left\| u_0 \xxe
\right\|}_{L^2(\O)} \to{\left\Vert u_0\xy\right\Vert}_{L^2(\O\times Y^m)}}$ are
said to be admissible for the two-scale (cut-and-projection)
convergence. In particular, classes of functions in
$L^2(\O,C_\sharp(Y^m))$ (dense subset in $L^2(\O\times Y^m)$) are
admissible.

In order to homogenize PDEs we need to identify the differential relationship between
$\chi$ and $u_0$, given a bounded sequence $(u_\e)$ in $W^{1,2}(\O)$
(such that $u_\e\wdto u_0$ and $\nabla u_\e\wdto \chi$).
This problem has been solved by
Allaire in the case of periodic functions \cite{Allaire1992}
and Bouchitt\'e et al. for quasiperiodic functions
\cite{Bouchitte+etal2010}.
In the latter case, the oscillations of the sequence
$(\nabla u_\e)$ cannot be
represented in general as the usual gradient  of a periodic function.

In \cite{Bouchitte+etal2010}
 the two-scale cut-and-projection limit of the gradients
of bounded sequences in $W^{1,2}(\O)$ were characterized as in the following Proposition.

\begin{proposition}\label{prop:gradtwoscale}
Let ${\bf R}$ be a matrix satisfying \eqref{eq:criterion} and $(\ue)$ a
bounded sequence in $W^{1,2}(\O)$. Then, there exist $u_0 \in W^{1,2}(\O)$ and $ \vec{w} \in L^2(\O,{\cal L}_{\bf R})$, and a subsequence (still denoted by $(\ue)$) such that
\begin{equation*}
\ue\dto u_0(\vec{x}) \; , \mbox{ and } \nabla\ue\wdto\nabla u_0(\vec{x}) +  \vec{w} \xy
\end{equation*}
where
\begin{equation*}
{\cal L}_{\bf R} = \Bigl \{
 \vec{v} \in L^2_{\sharp}(Y^m; \R^n) \,\mid\,\  \vec{v}(\vec{y}) = \sum_{\vec{k}\in \Z^m \setminus \{\vec{0}\}} \lambda_{\vec{k}} \; {\bf
R}^T\vec{k} e^{2\pi i \vec{k} \cdot \vec{y}} \; , \; \lambda_{\vec{k}} \in \C \Bigr \} 
\end{equation*}
\end{proposition}
Note that the sequence $(\lambda_{\vec{k}})$ which appears in the definition of
${\cal L}_{{\bf R}}$ does not necessarily belong to $l^2$. It only satisfies
$\sum_{\vec{k}\in\Z^m\setminus\{{\bf 0}\}} {\mid\lambda_{\vec{k}}\mid}^2 \; {\mid {\bf R}^T\vec{k}\mid}^2 < +\infty
$.

 In Section~\ref{sec:Compactness} (Proposition~\ref{prop:grad-split}) the limit $\vec{w}$ will be further characterized which will recast  Proposition~\ref{prop:gradtwoscale} into a more explicit form. The counterparts for the convergence of divergence and curl operators will be stated and proved as well.
We begin with the definition of some appropriate function spaces and
differential operators associated with the cut-and-projection method.

\section{Definition of function spaces and associated cut-and-projection differential operators and some of their properties}\label{sec:Functionspaces}

To carry out the homogenization analysis of PDEs defined on quasiperiodic domains, we need to pass to the limit
when $\eta$ goes to zero in gradient-,
divergence- and curl-operators acting on solutions of PDEs. To do this we introduce some suitable function spaces.
 Assuming we are considering PDEs defined on domains $\Omega \subset \R^n$,  we consider a matrix ${\bf R}$
with $m$ rows and $n$ columns, \ie  ${\bf R}: \R^n \to \R^m$
satisfying \eqref{eq:criterion}.

Any $u\in {L^2(Y^m)}$ can be used to  define a function $ u_{\bf R}  \in {L^2(\Omega)}$, by the cut-and-projection operation,
 \begin{equation*}
 u_{\bf R}(\vec{x}) = u( {\bf R}\vec{x})
 \end{equation*}
\noindent
The mapping $\vec{y} = {\bf R}\vec{x}$ and the chain rule yields that partial derivatives transform as
\begin{equation}\label{eq:projected_partial_derivative}
\frac{\partial}{\partial x_i} = \sum_{j=1}^m{\bf R}_{ji}\frac{\partial}{\partial y_j}
\end{equation}

Hence, the gradient of $u_{\bf R}$  is given by
 \begin{equation*}
\nabla  u_{\bf R}(\vec{x}) =
\mathrm{grad}\; u_{\bf R}(\vec{x}) =
\left({\bf R}^T \mathrm{grad}_{\vec{y}}\right)  u_{\bf R}(\vec{x}) =
\left({\bf R}^T \nabla_{\vec{y}}\right)  u_{\bf R}(\vec{x})
 \end{equation*}
where $\nabla$ and $\mathrm{grad}$ denote the usual gradient operator in $\R^n $, and $\nabla_{\vec{y}}$ and $\mathrm{grad}_{\vec{y}}$ denote the gradient operator in the range of ${\bf R}$ in $\R^m$.
\noindent
The  divergence and curl of vector valued functions  on $\Omega$ are defined as
\begin{equation*}
\nabla \cdot \vec{u}_{\bf R}(\vec{x}) =
\mathrm{div}\;\vec{u}_{\bf R}(\vec{x}) =
\left({\bf R}^T \nabla_{\vec{y}}\right) \cdot \vec{u}_{\bf R}(\vec{x})
\end{equation*}
and for $n=3$, we get
\begin{equation*}
\nabla  \times \vec{u}_{\bf R}(\vec{x}) =
\mathrm{curl}\;\vec{u}_{\bf R}(\vec{x}) =
 \left({\bf R}^T \nabla_{\vec{y}}\right) \times  \vec{u}_{\bf R}(\vec{x})
\end{equation*}

We can use these representations to define ${\bf R}$-dependent gradient, divergence  and curl operators acting on functions defined on domains in $\R^m$. They are
\begin{definition}\label{def:R-differetnialoperators}
\begin{equation*}
 \nabla_{\bf R}\;u(\vec{y}) =
 \mathrm{grad}_{\bf R}\;u(\vec{y}) :=
 \left({\bf R}^T \nabla_{\vec{y}}\right)  u(\vec{y})
\end{equation*}
\begin{equation*}
\nabla_{\bf R}\cdot \vec{u}(\vec{y}) =
\mathrm{div}_{\bf R}\;\vec{u}(\vec{y}) :=
\left({\bf R}^T \nabla_{\vec{y}}\right) \cdot \vec{u}(\vec{y})
\end{equation*}
 for $n=3$,
\begin{equation*}
\nabla_{\bf R} \times \vec{u}(\vec{y}) =
\mathrm{curl}_{\bf R}\;\vec{u}(\vec{y}) :=
 \left({\bf R}^T \nabla_{\vec{y}}\right) \times  \vec{u}(\vec{y})
\end{equation*}
\end{definition}

\begin{remark}\label{rem:divR_and_curlR}
The gradient operator $\mathrm{grad}_{\bf R}$  is a directional derivative given by  the projection on $\R^n$ of the usual gradient in $\R^m$. The divergence and curl operators are obtained   using the same projection  in combination with the usual nabla rules.  The  $\mathrm{div}_{\bf R}$ operator can also be interpreted as the divergence in $\R^m$ of the $m-$component vector ${\bf R} \vec{u}$, \ie $\mathrm{div}_{\bf R}\; \vec{u} = \nabla_{\vec{y}} \cdot {\bf R} \vec{u}$.
We  notice also that the
$ \mathrm{curl}_{\bf R}$ operator has the following  structure
\begin{equation*}
\begin{aligned}
&\mathrm{curl}_{\bf R}\, \vec{u}(\vec{y})   = {\bf R}^T \times \mathrm{grad}_{\vec{y}} \vec{u}^T(\vec{y}) := \\
&
({\bf R}_2^T  \mathrm{grad}_{\vec{y}} u_3 - {\bf R}_3^T \mathrm{grad}_{\vec{y}} u_2 ,  {\bf R}_3^T  \mathrm{grad}_{\vec{y}} u_1 - {\bf R}_1^T \mathrm{grad}_{\vec{y}} u_3 ,{\bf R}_1^T  \mathrm{grad}_{\vec{y}} u_2 - {\bf R}_2^T \mathrm{grad}_{\vec{y}} u_1)^T
\end{aligned}
\end{equation*}
where ${\bf R}_i^T$ is the $i$:th row of the transposed matrix  ${\bf R}^T$.
 \end{remark}

We define the following function spaces associated with the differential operators defined above
\begin{equation}
\displaystyle
{\cal H}_{\sharp}(\mathrm{grad}_{\bf R},Y^m) = \displaystyle{\Bigl \{u\in
{L^2_\sharp(Y^m)} \; \mid \; \mathrm{grad}_{\bf R}\;u \in
{L^2_\sharp(Y^m; \R^n)}
\Bigr \}
}
\label{def:grad-Yspace}
\end{equation}
\begin{equation}
\displaystyle
{\cal H}_{\sharp}(\mathrm{div}_{\bf R}, Y^m) :=  \displaystyle{\Bigl \{ \vec{u}\in
{L^2_\sharp(Y^m; \R^n)} \; \mid \; \mathrm{div}_{\bf R}\;\vec{u}  \in
{L^2_\sharp(Y^m)}
\Bigr \}
}
\label{def:divR-Yspace}
\end{equation}
\begin{equation}
\displaystyle
{\cal H}_{\sharp}(\mathrm{curl}_{\bf R}, Y^m) := \displaystyle{\Bigl \{ \vec{u}\in
{L^2_\sharp(Y^m; \R^3 )} \; \mid \; \mathrm{curl}_{\bf R}\;\vec{u} \in
{L^2_\sharp(Y^m; \R^3 )}
\Bigr \}
}
\label{def:curl-Yspace}
\end{equation}
and
\begin{equation}
\displaystyle
{\cal H}_{\sharp}(\mathrm{div}_{{\bf R}_0}, Y^m) :=  \displaystyle{\Bigl \{ \vec{u}\in
{\cal H}_{\sharp}(\mathrm{div}_{\bf R}, Y^m) \; \mid \; \mathrm{div}_{\bf R}\;\vec{u} \equiv 0 \Bigr \}
}
\label{def:div_0-Yspace}
\end{equation}
\begin{equation}
\displaystyle
{\cal H}_{\sharp}(\mathrm{curl}_{{\bf R}_{\vec{0}}}, Y^m) :=  \displaystyle{\Bigl \{ \vec{u}\in
{\cal H}_{\sharp}(\mathrm{curl}_{\bf R}, Y^m) \; \mid \; \mathrm{curl}_{\bf R}\;\vec{u} \equiv \vec{0} \Bigr \}
}
\label{def:curl_0-Yspace}
\end{equation}

We will  use the following spaces
\begin{equation}
\displaystyle
{W^{1,2}}(\Omega) :=  \displaystyle{\Bigl \{ u\in {L^2(\Omega)} \; \mid \; \mathrm{grad}\;u(\vec{x}) \in
{L^2(\Omega; \R^n)}
\Bigr \}}
\label{def:grad-space}
\end{equation}
\begin{equation}
\displaystyle
{W^{1,2}_0}(\Omega) :=  \displaystyle{ \Bigl\{ u\in W^{1,2}(\Omega)\; \mid \;  u|_{\partial \Omega} = \vec{0}
\Bigr \}
}
\label{def:grad_0-space}
\end{equation}
\begin{equation}
\displaystyle
{\cal H}(\mathrm{div}, \Omega) :=  \displaystyle{\Bigl \{ \vec{u}\in
{L^2( \Omega; \R^n )} \; \mid \; \mathrm{div}\;\vec{u}\in
{L^2(\Omega)}
\Bigr \}
}
\label{def:div-space}
\end{equation}
\begin{equation}
\displaystyle
{\cal H}(\mathrm{curl},\Omega) :=  \displaystyle{\Bigl \{ \vec{u}\in
{L^2(\Omega; \R^3 )} \; \mid \; \mathrm{curl}\;\vec{u}  \in
{L^2(\Omega; \R^3 )}
\Bigr \}
}
\label{def:curl-space}
\end{equation}
\begin{equation}
\displaystyle
{\cal H}_0(\mathrm{curl},\Omega) :=  \displaystyle{\Bigl \{ \vec{u}\in
{\cal H}(\mathrm{curl},\Omega) \; \mid \; \hat{\vec{\nu}}\times \vec{u}|_{\partial \Omega} = \vec{0}
\Bigr \}
}
\label{def:curl0-space}
\end{equation}
where $\hat{\vec{\nu}}$ is the unit normal to the boundary, $\partial \Omega$.
We have the following lemma which states that any $\mathrm{curl}_{\bf R}$-free
  vector field is given by a $\mathrm{grad}_{\bf R}$ of some potential.
  \begin{lemma}\label{lem:curlfree}
Any function $\vec{u} \in {\cal H}_{\sharp}(\mathrm{curl}_{{\bf R}_{\vec{0}}}, Y^m)$ is given as $\vec{u} = \mathrm{grad}_{\bf R}\;  \phi(\vec{y})$  for some scalar potential $\phi$ defined on $Y^m$.
\end{lemma}
\begin{proof}
Any $u\in {L^2(Y^m)}$ can be represented by a Fourier series,
\begin{equation*}
u(\vec{y}) = \sum_{\vec{k}\in\Z^m}
 u_{\vec{k}}
e^{2\pi i \vec{k}\cdot \vec{y}} ,   \qquad  u_{\vec{k}}\in \C
\end{equation*}
and the ${L^2_\sharp(Y^m; \R^n)}$-bounded  ${\bf R}$-gradient,   $\mathrm{grad}_{\bf R}\,\phi$, of any potential $\phi$, bounded or not in ${L^2_\sharp(Y^m)}$, can similarly be represented by  its Fourier series,
\begin{equation*}
\mathrm{grad}_{\bf R}\;  \phi(\vec{y}) =  \sum_{\vec{k}\in\Z^m}
 2\pi i {\bf R}^T\vec{k}   \phi_{\vec{k}}
e^{2\pi i  \vec{k}\cdot \vec{y}} ,   \qquad \phi_{\vec{k}}\in \C
\end{equation*}
and all vector valued function $\vec{u} \in {\cal H}_{\sharp}(\mathrm{curl}_{{\bf R}_{\vec{0}}}, Y^m)$ can be  represented analogously
\begin{equation*}
\vec{u}(\vec{y}) = \sum_{\vec{k}\in\Z^m}
 \vec{u}_{\vec{k}}
e^{2\pi i \vec{k}\cdot \vec{y}},   \qquad  u_{\vec{k}}\in \C^3
\end{equation*}
such that
\begin{equation*}
\mathrm{curl}_{\bf R}\;\vec{u}(\vec{y}) = \sum_{\vec{k}\in\Z^m}
 2\pi i {\bf R}^T\vec{k} \times \vec{u}_{\vec{k}}
e^{2\pi i \vec{k}\cdot \vec{y}} = \vec{0}
\end{equation*}
Since trigonometric functions are complete in $ {L^2(Y^m)}$, every Fourier coefficient of
$ \mathrm{curl}_{\bf R}\;\vec{u}(\vec{y}) $ is zero
\begin{equation*}
 2\pi i {\bf R}^T\vec{k} \times \vec{u}_{\vec{k}}  =  \vec{0}
\end{equation*}
Here,  vectors  $ {\bf a} =  2\pi i {\bf R}^T\vec{k}$  and $ {\bf b} =  \vec{u}_{\vec{k}}$ are in $\R^3$ and $\C^3$, respectively.
For any 3D vectors $ {\bf a},  {\bf b} $, such that $ {\bf a} \times  {\bf b}  = \vec{0}$,  vector ${\bf b}$ can be written as
$ {\bf b} = C  {\bf a} $ for a complex scalar constant $C$. Hence,   $ \vec{u}_{\vec{k}} $ can be represented as
\begin{equation*}
 \vec{u}_{\vec{k}}  = 2\pi i {\bf R}^T\vec{k}   \phi_{\vec{k}}
\end{equation*}
for a complex scalar $  \phi_{\vec{k}}\in \C$. This gives
\begin{equation*}
\vec{u}(\vec{y}) = \sum_{\vec{k}\in\Z^m}
 \vec{u}_{\vec{k}}
e^{2\pi i \vec{k}\cdot \vec{y}}    =
 \sum_{\vec{k}\in\Z^m}  2\pi i {\bf R}^T\vec{k}   \phi_{\vec{k}}  e^{2\pi i \vec{k}\cdot \vec{y}}  =
 \mathrm{grad}_{\bf R}\;  \phi(\vec{y})
\end{equation*}
\end{proof}
\begin{remark}
Note that it is not necessary that $\phi  \in L^2(Y^m)$ to make $\mathrm{grad}_{\bf R}\;  \phi $ bounded in  $ L^2(Y^m; \R^n)$.
\end{remark}
We are now in the position to   state that spaces ${\cal H}_{\sharp}(\mathrm{curl}_{{\bf R}_{\vec{0}}}, Y^m)$  and ${\cal H}_{\sharp}(\mathrm{div}_{{\bf R}_0}, Y^m)$ are orthogonal to curls and gradients, respectively. We define the space of curls as
\begin{equation*}
{\cal M}_{\bf R_\sharp}(Y^m, \R^3 ) = \Bigl \{ \vec{w} \in
{L^2_\sharp(Y^m; \R^3 )} \; \mid \;    \vec{w}  (\vec{y}) =  \mathrm{curl}_{\bf R}\; \vec{u}(\vec{y}) \Bigr\}
\end{equation*}
for some vector valued function $\vec{u}$ in $\R^3$ defined on $Y^m$, and the space of ${\bf R}$-gradients
\begin{equation*}
{\cal L}_{\bf R_\sharp}(Y^m, \R^3 ) = \Bigl \{ \vec{w} \in
{L^2_\sharp(Y^m; \R^3 )} \; \mid \;    \vec{w}   (\vec{y}) =
 \mathrm{grad}_{\bf R} \; \phi(\vec{y})
 \Bigr \}
\end{equation*}
for some scalar potential $\phi$ defined on $Y^m$.

\begin{lemma}\label{lem:orthogonality}
We have the following orthogonal decompositions of $L^2_\sharp(Y^m; \R^3 )$
\begin{itemize}
\item[i)] $
L^2_\sharp(Y^m; \R^3 ) =  {\cal H}_{\sharp}(\mathrm{curl}_{{\bf R}_{\vec{0}}}, Y^m) \oplus {\cal M}_{\bf R_\sharp}(Y^m, \R^3 )
$
\item[ii)] $L^2_\sharp(Y^m; \R^3 ) = {\cal H}_{\sharp}(\mathrm{div}_{{\bf R}_0}, Y^m) \oplus
{\cal L}_{\bf R_\sharp}(Y^m, \R^3 ) $
\end{itemize}
\end{lemma}
\begin{proof} {\it i)}
For any $\vec{w} \in\
{L^2_\sharp(Y^m; \R^3 )}$  orthogonal to ${\cal H}_{\sharp}(\mathrm{curl}_{{\bf R}_{\vec{0}}}, Y^m)$  we have due to
Lemma~\ref{lem:curlfree}
and Parseval's relation
\begin{equation}\label{eq:i-orthogonal}
\int_{Y^m} \vec{v}(\vec{y}) \cdot \vec{w}(\vec{y}) \; \mathrm{d}\vec{y} =  \int_{Y^m} \mathrm{grad}_{\bf R} \phi (\vec{y}) \cdot  \vec{w}(\vec{y}) =  \sum_{\vec{k}\in\Z^m\setminus\{\vec{0}\}} 2\pi i {\bf R}^T\vec{k}  \phi_{\vec{k}}  \cdot \overline{\vec{w}_{\vec{k}}} = 0
\end{equation}
for all $ \vec{v} \in {\cal H}_{\sharp}(\mathrm{curl}_{{\bf R}_{\vec{0}}}, Y^m)$, where $\overline{\vec{w}}$ denotes the complex conjugate of $\vec{w}$. Since \eqref{eq:i-orthogonal} holds for all $\vec{v} \in {\cal H}_{\sharp}(\mathrm{curl}_{{\bf R}_{\vec{0}}}, Y^m)$ it must also hold term vise for all ${\vec{k}\in\Z^m\setminus\{\vec{0}\}}$. It follows that
$$
2\pi i \left({\bf R}^T\vec{k}\right) \cdot \overline{\vec{w}_{\vec{k}}} = 0
$$
for all ${\vec{k}\in\Z^m\setminus\{\vec{0}\}}$. We conclude that this orthogonality identity is true if and only if ${\bf R}^T\vec{k}$ and  $\overline{\vec{w}_{\vec{k}}}$ are orthogonal, \ie  we can write
$$
\overline{\vec{w}_{\vec{k}}} = \overline{ 2\pi i {\bf R}^T\vec{k} \times \vec{u}_{\vec{k}}}
$$
for some vector $\vec{u}_{\vec{k}}$ in $\C^3$, which defines
$\vec{w} \in {\cal M}_{\bf R_\sharp}(Y^m, \R^3 )$,  where $\vec{w}= \mathrm{curl}_{\bf R} \; \vec{u}$. Now, it remains to prove that  ${\cal H}_{\sharp}(\mathrm{curl}_{{\bf R}_{\vec{0}}}, Y^m)$ is the only space orthogonal to ${\cal M}_{\bf R_\sharp}(Y^m, \R^3 )$.
Assume $\vec{v} \in\
{L^2_\sharp(Y^m; \R^3 )}$  and is orthogonal to  ${\cal M}_{\bf R_\sharp}(Y^m, \R^3 )$, \ie
\begin{equation*}
\int_{Y^m} \vec{v}(\vec{y}) \cdot \vec{w}(\vec{y}) \; \mathrm{d}\vec{y} =  \sum_{\vec{k}\in\Z^m\setminus\{\vec{0}\}}   \vec{v}_{\vec{k}}  \cdot \overline{\vec{w}_{\vec{k}}}=0
\end{equation*}
for all $\vec{w} \in {\cal M}_{\bf R_\sharp}(Y^m, \R^3 )$. This implies
$$
 \vec{v}_{\vec{k}}  \cdot \overline{ 2\pi i {\bf R}^T\vec{k} \times \vec{u}_{\vec{k}}}= 0
$$
for all ${\vec{k}\in\Z^m\setminus\{\vec{0}\}}$. That is, $\vec{v}_{\vec{k}} = \phi_{\vec{k}} {\bf R}^T\vec{k}$  for some complex scalar $\phi_{\vec{k}}$, which defines $\vec{v}$ as a gradient $\mathrm{grad}_{\bf R} \; \phi\in\
{L^2_\sharp(Y^m; \R^3 )}$. We conclude $\vec{v} \in {\cal H}_{\sharp}(\mathrm{curl}_{{\bf R}_{\vec{0}}}, Y^m)$. \\
{\it ii)}
For any fixed $\vec{w} \in
{L^2_\sharp(Y^m; \R^3 )}$ orthogonal to ${\cal H}_{\sharp}(\mathrm{div}_{{\bf R}_0}, Y^m)$ we have by assumption and  Parseval's relation
\begin{equation*}
\begin{aligned}
\int_{Y^m} \vec{u}(\vec{y})  \cdot \vec{w}(\vec{y})  \; \mathrm{d}\vec{y} = \sum_{\vec{k}\in\Z^m\setminus\{\vec{0}\}}  \vec{u}_{\vec{k}} \cdot \overline{\vec{w}_{\vec{k}}}= 0
   \end{aligned}
\end{equation*}
for all $\vec{u} \in {\cal H}_{\sharp}(\mathrm{div}_{{\bf R}_0}, Y^m)$.
This implies that
\begin{equation*}\label{eq:orthogonality_div0}
 \vec{u}_{\vec{k}} \cdot \overline{\vec{w}_{\vec{k}}} =  0
\end{equation*}
for all  $\vec{u} \in {\cal H}_{\sharp}(\mathrm{div}_{{\bf R}_0}, Y^m)$ and all ${\vec{k}\in\Z^m\setminus\{\vec{0}\}}$. We also have
$$
{\bf R}^T\vec{k} \cdot \vec{u}_{\vec{k}} = 0
$$
for each fixed  ${\vec{k}\in\Z^m\setminus\{\vec{0}\}}$ and all  $\vec{u}\in {\cal H}_{\sharp}(\mathrm{div}_{{\bf R}_0}, Y^m)$. This implies that
 $\overline{\vec{w}_{\vec{k}}}$ is parallel with  ${\bf R}^T\vec{k}$. That is
$$
\overline{\vec{w}_{\vec{k}}} = \overline{2\pi i {\bf R}^T\vec{k} \phi_{\vec{k}}}
$$
for some {complex} scalar  $\phi_{\vec{k}}$.
This defines  $\vec{w} \in  {\cal L}_{\bf R_\sharp}(Y^m, \R^3 )$.
Further, let  $\vec{u} \in {L^2_\sharp(Y^m; \R^3 )}$ and assume
\begin{equation*}
\int_{Y^m} \vec{u}(\vec{y})  \cdot \vec{w}(\vec{y})  \; \mathrm{d}\vec{y} = \sum_{\vec{k}\in\Z^m\setminus\{\vec{0}\}}  \vec{u}_{\vec{k}} \cdot \overline{\vec{w}_{\vec{k}}}= 0
\end{equation*}
{for all
$\vec{w} \in  {\cal L}_{\bf R_\sharp}(Y^m, \R^3 )$. That is }
\begin{equation*}
 \vec{u}_{\vec{k}} \cdot \overline{2\pi i {\bf R}^T\vec{k} \phi_{\vec{k}}}=  0
\end{equation*}
{for all ${\vec{k}\in\Z^m\setminus\{\vec{0}\}}$, \ie}
\begin{equation*}
 \left(\vec{u}_{\vec{k}} \cdot \overline{2\pi i {\bf R}^T\vec{k} }\right) \overline{\phi_{\vec{k}}}=  0
\end{equation*}
{for all ${\vec{k}\in\Z^m\setminus\{\vec{0}\}}$ and all $\vec{w} \in  {\cal L}_{\bf R_\sharp}(Y^m, \R^3 )$, {thus} $\vec{u}\in {\cal H}_{\sharp}(\mathrm{div}_{{\bf R}_0}, Y^m)$, which completes the proof.}
\end{proof}

\section{Compactness results}\label{sec:Compactness}
In the following main compactness results we let ${\bf R}$ be a matrix with $m$ rows and $3$ columns ($m>3$)
satisfying \eqref{eq:criterion}  and recall that $\O$ is an open bounded set
of $\R^3$ with a Lipschitz boundary.
We begin by recasting  Proposition \ref{prop:gradtwoscale} into the following familiar form.
\begin{proposition}\label{prop:grad-split}
Let  $\{u_\e\}$ be a uniformly bounded
sequence in $W^{1,2}(\Omega)$.
Then there exist a subsequence $\{u_{\e_k}\}$
and  functions $u\in W^{1,2}(\Omega)$ and $\mathrm{grad}_{\bf R} \;  {u}_1\xy \in L^2(\Omega, L^2_\sharp(Y^m; \R^3 ))$,
such that
\begin{equation}\label{eq:gradsplit_2s_strong}
u_{\e_k} \dto u(\vec{x})
\end{equation}
\begin{equation}
\mathrm{grad}\; u_{\e_k} \wdto \mathrm{grad} \; u (\vec{x}) + \mathrm{grad}_{\bf R} \;  u_1\xy ,  \qquad
{\e_k} \to 0
\end{equation}
\end{proposition}
\begin{proof}
The proof is identical to the one for Proposition \ref{prop:gradtwoscale} given in \cite{Bouchitte+etal2010} or as in the proof of how gradient splits in two-scale convergence,  proved in \cite{Allaire1992}, using Lemma~\ref{lem:orthogonality}.
The limit \eqref{eq:gradsplit_2s_strong} follows at once due to the a priori estimate. Further, we have a two-scale cut-and-projection limit
\begin{equation*}
\mathrm{grad}\; {{u}_\e} \wdto  \vec{\chi}\xy
\end{equation*}
Choosing  test functions $\phi_1 \in  W_0^{1,2}(\Omega)$,
$\phi_2 \in {\cal H}_{\sharp}(\mathrm{div}_{{\bf R}_0}, Y^m) $
yields
\begin{equation*}
        \begin{aligned}
& \lim _{\e\to 0} \IOm \mathrm{grad}\; {u}_\e (\vec{x})  \cdot \phi_1(\vec{x}) \phi_2\left(\frac{{\bf R}\vec{x}}{\e}\right) \,\mathrm{d}\vec{x} = 
- \lim _{\e\to 0} \IOm  {u}_\e (\vec{x}) \; \mathrm{div}\; \phi_1(\vec{x}) \phi_2\left(\frac{{\bf R}\vec{x}}{\e}\right) \,\mathrm{d}\vec{x} = \\ &
- \IOAY   u (\vec{x}) \; \mathrm{div}\;  \phi_1(\vec{x}) \phi_2({\bf y})\,\mathrm{d}\vec{x}\mathrm{d}\vec{y} \; = 
\IOAY \mathrm{grad} \; u (\vec{x}) \cdot \phi_1(\vec{x}) \phi_2({\bf y})\,\mathrm{d}\vec{x}\mathrm{d}\vec{y} \;
   \end{aligned}
\end{equation*}
\noindent
Now,  $\vec{w} = \vec{\chi}\xy  - \mathrm{grad}\; {u}(\vec{x})$ belongs to the space orthogonal to ${\cal H}_{\bf R_\sharp}(\mathrm{div}_0, Y^m)$. Lemma~\ref{lem:orthogonality} gives
\begin{equation*}
\vec{\chi}\xy  - \mathrm{grad}\; {u}(\vec{x}) = \mathrm{grad}_{\bf R} \;  u_1\xy
\end{equation*}
for some gradient $\mathrm{grad}_{\bf R} \;   {u}_1\xy \in L^2(\Omega, L^2_\sharp(Y^m; \R^3 ))$.
\end{proof}
For the counterpart  of the two-scale limit of curls we need the following lemma stating a differential identity.
\begin{lemma}\label{lem:curlgrad} Any  $\phi \in {\cal H}(\mathrm{grad},\Omega) \cap  {\cal H}(\mathrm{grad}_{\bf R},Y^m)$ satisfies
\begin{equation}\label{eq:curlgrad}
\mathrm{curl}_{\vec{x}} \mathrm{grad}_{\bf R} \;  \phi\xy =
- \mathrm{curl}_{\bf R} \mathrm{grad}_{\vec{x}} \;  \phi\xy
\end{equation}
\end{lemma}
\begin{proof} It is easily verified that the following analogous identity holds for the standard differential operators
\begin{equation*}
\mathrm{curl}_{\vec{x}} \mathrm{grad}_{\vec{y}} \;  \phi\xy =
-  \mathrm{curl}_{\vec{y}} \mathrm{grad}_{\vec{x}} \;  \phi\xy
\end{equation*}
for all sufficiently differentiable functions $\phi$.
Let  $\phi \in {\cal H}(\mathrm{grad},\Omega) \cap  {\cal H}(\mathrm{grad}_{\bf R},Y^m)$. Assume $\Omega$ is bounded. The Fourier series of $
\mathrm{curl}_{\vec{x}} \mathrm{grad}_{\bf R} \;  \phi\xy$ gives
\begin{equation*}
 2\pi i  2\pi i \vec{\xi} \times {\bf R}^T\vec{k}  \phi_{\vec{\xi}\vec{k}}  =
-  2\pi i  2\pi i {\bf R}^T\vec{k} \times  \vec{\xi} \phi_{\vec{\xi}\vec{k}}
\end{equation*}
which obviously holds
for all  $\vec{\xi} \in\Z^3 $,  ${\vec{k}\in\Z^m\setminus\{\vec{0}\}}$.
For $\Omega$ unbounded we use the Fourier transform and $\hat{\phi}_{\vec{k}}(\vec{\xi})$ instead of $\phi_{\vec{\xi}\vec{k}}$. This proves  \eqref{eq:curlgrad}.
\end{proof}
\begin{proposition}\label{prop:curl_split}
Let  $\{\vec{u}_\e\}$ be a uniformly bounded
sequence in $H(\mathrm{curl},\Omega)$.
Then there exist a subsequence
 $\{\vec{u}_{\e_k}\}$
and  functions
{$\vec{u}_0\in {\cal H}(\mathrm{curl},\Omega,{\cal H}_{\sharp}(\mathrm{curl}_{{\bf R}_0},Y^m ))$,}\\
$\mathrm{grad}_{\bf R} \;  \phi   \in  L^2(\O,L^2_\sharp(Y^m; \R^3))$,
and
$\mathrm{curl}_{\bf R} \;  \vec{u}_1 \in L^2(\O,L^2_\sharp(Y^m; \R^3))$
such that
\begin{equation}\label{eq:curlsplit_2s_weak}
\vec{u}_{\e_k} \wdto \vec{u}_0(\vec{x},\vec{y}) =  \vec{u}(\vec{x}) +  \mathrm{grad}_{\bf R} \;  \phi\xy
\end{equation}
\begin{equation}
\mathrm{curl}\; \vec{u}_{\e_k} \wdto \mathrm{curl} \; \vec{u} (\vec{x}) + \mathrm{curl}_{\bf R} \;  \vec{u}_1\xy
\end{equation}
as ${\e_k} \to 0$,
where
\begin{equation*}
\vec{u}(\vec{x}) = \int_{Y^m} \vec{u}_0(\vec{x},\vec{y}) \; \mathrm{d}\vec{y}
\end{equation*}
\end{proposition}

\begin{proof}
The proof follows the line of the proof of Proposition~7 in \cite{Wellander2009}, or by using Lemma~\ref{lem:orthogonality} similarly as in Proposition~\ref{prop:grad-split}, as follows.
The limit \eqref{eq:curlsplit_2s_weak} follows  due to the a priori estimate and standard arguments, using admissible test functions $ \vec{\phi}(\vec{x},\vec{y})$ and  $\eta \vec{\phi}(\vec{x},\vec{y})$ for the sequences $\{ \vec{u}_\e \}$ and $\{ \mathrm{curl}\; \vec{u}_\e \}$, respectively.
Further, we have a two-scale limit
\begin{equation*}
\mathrm{curl}\; \vec{u}_\e \wdto  \vec{\chi}\xy
\end{equation*}
Using test functions in ${\cal H}_0(\mathrm{curl},\Omega)$ gives
\begin{equation}\label{eq:weak_curl_limit}
\mathrm{curl}\; \vec{u}_\e \wdto \mathrm{curl}\; \vec{u}
 \end{equation}
Next we choose test functions $\phi_1 \in  W^{1,2}(\Omega)$,
$\vec{\phi}_2 \in{\cal H}_{\sharp}(\mathrm{curl}_{\bf R_0}, Y^m)$ which by using \eqref{eq:curlsplit_2s_weak} and  \eqref{eq:weak_curl_limit} gives
\begin{equation}
\label{eq:toto1}
\begin{aligned}
&
\lim _{\e\to 0} \IOm \mathrm{curl}\; \vec{u}_\e  (\vec{x})  \cdot \phi_1(\vec{x}) \vec{\phi}_2\left(\frac{ {\bf R}\vec{x}}{\e}\right) \,\mathrm{d}\vec{x} =
\\&
\lim _{\e\to 0} \IOm  \vec{u}_\e  (\vec{x})  \cdot \mathrm{curl}\; \left(\phi_1(\vec{x}) \vec{\phi}_2\left(\frac{ {\bf R}\vec{x}}{\e}\right) \right) \,\mathrm{d}\vec{x} =\\
 &\IOAY \left( \vec{u} (\vec{x}) + \mathrm{grad}_{\bf R} \;  \phi\xy \right) \cdot \mathrm{curl}\; \left(\phi_1(\vec{x}) \vec{\phi}_2({\bf y})\right) \,\mathrm{d}\vec{x}\mathrm{d}\vec{y}=\\
 &\IOAY \left(\mathrm{curl} \; \vec{u}  (\vec{x}) +  \mathrm{curl}_{\vec{x}} \;\mathrm{grad}_{\bf R} \;  \phi\xy\right)  \cdot \phi_1(\vec{x}) \vec{\phi}_2({\bf y})\,\mathrm{d}\vec{x}\mathrm{d}\vec{y}
\end{aligned}
\end{equation}
 Lemmas~\ref{lem:curlgrad}, \ref{Integration_by_parts_Stokes} (see Appendix) and the choice of test functions   yields
\begin{equation}
\label{eq:toto2}
\begin{aligned}
& \IOAY  \mathrm{curl}_{\vec{x}} \;\  \mathrm{grad}_{\bf R} \;  \phi\xy   \cdot \phi_1(\vec{x}) \vec{\phi}_2({\bf y})\,\mathrm{d}\vec{x}\mathrm{d}\vec{y} = \\
& - \IOAY   \mathrm{curl}_{\bf R} \; \mathrm{grad}_{\vec{x}} \;  \phi\xy   \cdot \phi_1(\vec{x}) \vec{\phi}_2({\bf y})\,\mathrm{d}\vec{x}\mathrm{d}\vec{y} = \\
& - \IOAY  \mathrm{grad}_{\vec{x}}\;  \phi\xy   \cdot \phi_1(\vec{x}) \mathrm{curl}_{\bf R} \; \vec{\phi}_2({\bf y})\,\mathrm{d}\vec{x}\mathrm{d}\vec{y} = 0
\end{aligned}
\end{equation}
Combining (\ref{eq:toto1}) and (\ref{eq:toto2}),  lemma~\ref{lem:orthogonality} yields
\begin{equation*}
 \vec{\chi}\xy  - \mathrm{curl}\; \vec{u}(\vec{x}) = \mathrm{curl}_{\bf R} \;  \vec{u}_1\xy
\end{equation*}
for some function
$\mathrm{curl}_{\bf R} \;  \vec{u}_1 \in L^2(\O,L^2_\sharp(Y^m; \R^3))$.
\end{proof}

\begin{remark}
Inspired by \cite{Cherednichenko+Cooper2016,Kamotski+Smyshlyaev2013} we  decompose  $W^{1,2}_\sharp(Y^m)$ into two orthogonal spaces,
\begin{equation*}
W^{1,2}_\sharp(Y^m) = X \oplus  X^\perp
\end{equation*}
where
\begin{equation}\label{eq:W12_composition}
    X = \left\{ u \in W^{1,2}_\sharp(Y^m) | {\bf R}{\bf R}^T \; \nabla_y u = 0 \right\}
\end{equation}
and
\begin{equation}\label{eq:X-proj_space}
X^\perp = \left\{ u \in W^{1,2}_\sharp(Y^m) | \left({\bf I}_m - {\bf R}{\bf R}^T\right) \; \nabla_y u = 0 \right\}
\end{equation}
 The columns of ${\bf R}$ are eigenvectors to ${\bf R}{\bf R}^T$, with eigenvalues $\{\lambda_1, \dots, \lambda_n\}$ that all take the value one. The vector space in the kernel of ${\bf R}{\bf R}^T$ is orthogonal to these eigenenvectors.   Matrix $\left({\bf I}_m - {\bf R}{\bf R}^T\right)$ projects vectors to the subspace of $Y^m$  orthogonal to the hyperplane defined by ${\bf R}$, which characterizes completely the degeneracy of ${\bf R}{\bf R}^T$.
\end{remark}

\begin{lemma}\label{lem:R-poisson_eqn}
 The   equation
\begin{equation}\label{eq:projected_Poisson}
 - \mathrm{div}_{\bf R} \;  \mathrm{grad}_{\bf R}\;{\theta}\xy  = f\xy
\end{equation}
has a unique solution in $X^\perp$.
\end{lemma}
\begin{proof}
 Due to Definition~\ref{def:R-differetnialoperators} and Remark~\ref{rem:divR_and_curlR} equation \eqref{eq:projected_Poisson} can be written as
  \begin{equation}\label{eq:projected_Poisson_2}
 - \mathrm{div}_{\vec{y}} \; {\bf R}{\bf R}^T   \mathrm{grad}_{\vec{y}}\; {\theta}\xy  = f\xy
\end{equation}
 Excluding all potentials with gradient components in $\mathrm{Ker }\; {\bf R}{\bf R}^T$, which is the same as looking for solutions with gradients in  $\mathrm{Ker }\; \left({\bf I}_m - {\bf R}{\bf R}^T\right)$, makes $ {\bf R}{\bf R}^T $ coercive on $X^\perp$ and the  problem well posed.
\end{proof}

\begin{remark}
Equation \eqref{eq:projected_Poisson} corresponds to solving Poisson's equation in the projected hyperplane. This  restricts  the solutions to a subspace of $W^{1,2}_\sharp(Y^m)$  as in \cite{Cherednichenko+Cooper2016,Kamotski+Smyshlyaev2013}.  Another option, and maybe from a numerical point of view simpler way to handle the degenerated elliptic equation is similar to what was done in \cite{Blanc+etal2015}. Regularize the equation by  adding  a term, e.g.   $(- \alpha \; \mathrm{div}_{\vec{y}}\; \left({\bf I}_m - {\bf R}{\bf R}^T\right) \; \nabla_{\vec{y}} u)$, $\alpha >0$.
The physical interpretation of this regularization is that we allow transport perpendicular to the hyperplane on which ${\bf R}^T$
projects in the higher dimensional space. This corresponds to shortcut   points in real space, $\R^n$, which can be arbitrarily far from each other but   arbitrarily close in $\R^m$.
\end{remark}

\begin{proposition}\label{prop:div_split}
Let  $\{\vec{u}_\e\}$ be a uniformly bounded
sequence in $H(\mathrm{div},\Omega)$.
Then there exist a subsequence
 $\{\vec{u}_{\e_k}\}$ 
and functions
{$\vec{u}_0\in {\cal H}(\mathrm{div},\Omega,{\cal H}_{\sharp}(\mathrm{div}_{{\bf R}_0},Y^m ))$}
and
{$\mathrm{div}_{\bf R} \; \vec{u}_1 \in L^2(\O,L^2_\sharp(Y^m))$}
such that
\begin{equation}\label{eq:divsplit_2s_weak}
\vec{u}_{\e_k} \wdto \vec{u}_0(\vec{x},\vec{y})
\end{equation}
\begin{equation}\label{eq:divsplit_2s}
\mathrm{div}\; \vec{u}_{\e_k} \wdto  \mathrm{div} \; \vec{u}(\vec{x}) + \mathrm{div}_{\bf R} \; \vec{u}_1\xy
\end{equation}
as ${\e_k} \to 0$, where
\begin{equation*}
u(\vec{x}) = \int_{Y^m}  u_0(\vec{x},\vec{y}) \; \mathrm{d}\vec{y}
\end{equation*}
\end{proposition}
\begin{proof}
The limit \eqref{eq:divsplit_2s_weak} follows at once and we have  two-scale limits
\begin{equation*}
\e\mathrm{div}\; \vec{u}_\e \wdto \mathrm{div}_{\bf R} \; \vec{u}_0\xy = 0,
\end{equation*}
\begin{equation*}
\mathrm{div}\; \vec{u}_\e \wdto  \chi\xy
\end{equation*}
Using test functions $\phi\in W_0^{1,2}(\Omega)$ gives
\begin{equation*}
\mathrm{div}\; \vec{u}_\e \wdto \mathrm{div}\; \vec{u} =  \int_{Y^m}  u_0(\vec{x},\vec{y}) \; \mathrm{d}\vec{y}
 \end{equation*}
The difference defines, for fixed $\vec{x}$  in $\R^3$, a function
\begin{equation*}
f\xy  = \chi\xy  - \mathrm{div}\; \vec{u}(\vec{x})
\end{equation*}
Due to Lemma~\ref{lem:R-poisson_eqn}, the equation
\begin{equation*}
 - \mathrm{div}_{\bf R} \;  \mathrm{grad}_{\bf R}\;{\theta}\xy  = f\xy
\end{equation*}
 has  a unique solution in
 $X^\perp$  for a.e. $\vec{x} \in \R^3$.
 Defining
 \begin{equation*}
 \vec{u}_1\xy = - \mathrm{grad}_{\bf R}\;{\theta}\xy
\end{equation*}
 yields the limit \eqref{eq:divsplit_2s}.
\end{proof}
We have the following analogue of Proposition~1.14.(ii) in \cite{Allaire1992}.
\begin{proposition}\label{prop:eta_grad}
Let  $\{u_\e\}$ and $\{\e\mathrm{grad}\;{u}_\e\}$ be two uniformly bounded
sequences in $L^2(\Omega)$ and $L^2(\Omega;\R^3)$, respectively.
Then there exist a subsequence
 $\{{u}_{\e_k}\}$
and a function {${u}_0\in L^2(\O,{\cal H}_{\sharp}(\mathrm{grad}_{\bf R} ,Y^m))$}
such that
\begin{equation}\label{eq:u_2s_weak}
u_{\e_k} \wdto u_0(\vec{x},\vec{y})\; \qquad \in L^2(\O,L^2_{\sharp}(Y^m))
\end{equation}
\begin{equation}\label{eq:etagradu_2s}
{\e_k} \, \mathrm{grad}\; u_{\e_k} \wdto  \mathrm{grad}_{\bf R}  \; {u}_0\xy  \; \in L^2(\O,L^2_{\sharp}(Y^m;\R^3))
\end{equation}
as ${\e_k} \to 0$.
\end{proposition}

\begin{proof}
The limit \eqref{eq:u_2s_weak} follows at once, using Proposition~\ref{prop:weaktwoscaleprop}. {We have }
\begin{equation*}
{\e_k} \, \mathrm{grad}\; u_\e \wdto  \vec{\chi}\xy
\end{equation*}

  This  limit is characterized by choosing test functions $\phi_1 \in  W^{1,2}(\Omega)$ and 
$\vec{\phi}_2 \in $  ${\cal H}_{\sharp}(\mathrm{div}_{\bf R}, Y^m)$ which gives
\begin{equation*}
\begin{aligned}
 & \IOAY  \vec{\chi}\xy \cdot \phi_1(\vec{x}) \vec{\phi}_2({\bf y})\,\mathrm{d}\vec{x}\mathrm{d}\vec{y} = \\ &
\lim _{{\e_k}\to 0} \IOm {\e_k} \,\mathrm{grad}\; \vec{u}_{\e_k}  (\vec{x})  \cdot \phi_1(\vec{x}) \vec{\phi}_2\left(\frac{ {\bf R}\vec{x}}{{\e_k}}\right) \,\mathrm{d}\vec{x} = \\
&-\lim _{{\e_k}\to 0} \IOm \vec{u}_{\e_k}  (\vec{x}) \, {\e_k}\,\mathrm{div}\; \left(\phi_1(\vec{x}) \vec{\phi}_2\left(\frac{ {\bf R}\vec{x}}{{\e_k}}\right)\right) \,\mathrm{d}\vec{x} = \\
& -\lim_{{\e_k}\to 0} \IOm \vec{u}_{\e_k}  (\vec{x}) \left[ {\e_k} \, \mathrm{grad}\; \phi_1(\vec{x})   \cdot \vec{\phi}_2\left(\frac{ {\bf R}\vec{x}}{{\e_k}}\right) + \phi_1(\vec{x})\,\mathrm{div}_{\bf R}\;   \vec{\phi}_2\left(\frac{ {\bf R}\vec{x}}{{\e_k}}\right) \right]  \,\mathrm{d}\vec{x} = \\
& -\IOAY u_0\xy \phi_1(\vec{x}) \, \mathrm{div}_{\bf R} \;  \vec{\phi}_2({\bf y})\,\mathrm{d}\vec{x}\mathrm{d}\vec{y} = \\ &
\IOAY \left(\mathrm{grad}_{\bf R} \;  u_0\xy\right)  \cdot \phi_1(\vec{x}) \vec{\phi}_2({\bf y})\,\mathrm{d}\vec{x}\mathrm{d}\vec{y}
\end{aligned}
\end{equation*}
Then, $\vec{\chi}\xy =\mathrm{grad}_{\bf R} \;  u_0\xy$ almost everywhere in $\Omega\times Y^m$.
\end{proof}
Similarly, we have the following  convergence results of bounded sequences of curls and divergences of vector fields. The proofs are omitted since they are analogous to the proof of Proposition~\ref{prop:eta_grad}.

\begin{proposition}\label{prop:eta_curl}
Let  $\{\vec{u}_\e\}$ and $\{\e\mathrm{curl}\vec{u}_\e\}$ be two uniformly bounded
sequences in $L^2(\Omega;\R^3)$.
Then there exist a subsequence
 $\{\vec{u}_{\e_k}\}$
and a function
$\vec{u}_0$ in\\  $L^2(\O,{\cal H}_{\sharp}(\mathrm{curl}_{{\bf R}}, Y^m))$
such that
\begin{equation}\label{eq:vu_curl_2s_weak}
\vec{u}_{\e_k} \wdto \vec{u}_0(\vec{x},\vec{y})
\end{equation}
\begin{equation}\label{eq:etacurlu_2s}
{\e_k}\,\mathrm{curl}\; \vec{u}_{\e_k} \wdto  \mathrm{curl}_{\bf R} \; \vec{u}_0\xy
\end{equation}
as ${\e_k} \to 0$.
\end{proposition}

\begin{proposition}\label{prop:eta_div}
Let  $\{\vec{u}_\e\}$ and $\{\e\mathrm{div}\vec{u}_\e\}$ be two uniformly bounded
sequences in $L^2(\Omega;\R^3)$ and in $L^2(\Omega)$.
Then there exist a subsequence
 $\{\vec{u}_{\e_k}\}$
and a function {$\vec{u}_0\in  L^2(\O,{\cal H}_{\sharp}(\mathrm{div}_{{\bf R}}, Y^m)) $}
such that
\begin{equation}\label{eq:vvu_2s_weak}
\vec{u}_{\e_k} \wdto \vec{u}_0(\vec{x},\vec{y})
\end{equation}
\begin{equation}\label{eq:etadivu_2s}
{\e_k}\, \mathrm{div}\; \vec{u}_{\e_k} \wdto  \mathrm{div}_{\bf R} \; \vec{u}_0\xy
\end{equation}
as ${\e_k} \to 0$.
\end{proposition}

The following results are the two-scale cut-and-projection counterpart of the two-scale limits of divergence free and curl free bounded sequences, \eg see \cite{Allaire1992} for the divergence free case. The proofs follow the lines of Propositions~\ref{prop:grad-split} and \ref{prop:eta_grad}.

\begin{proposition}
\label{prop:eta_div_free}
Let  $\{\vec{u}_\e\}$  be a divergence free sequence, uniformly bounded in $L^2(\Omega;\R^3)$.

Then there exist a subsequence $\{\vec{u}_{\e_k}\}$
and a function {$\vec{u}_0\in  L^2(\O,{\cal H}_{\sharp}(\mathrm{div}_{{\bf R}_0},Y^m))$}
such that
\begin{equation*}
\vec{u}_{\e_k} \wdto \vec{u}_0(\vec{x},\vec{y})
\end{equation*}
which satisfies
\begin{equation}
\label{eq:div_R_free_limit}
         \begin{aligned}
& \mathrm{div}_{\bf R}\, \vec{u}_0(\vec{x},\vec{y}) = 0 \qquad \mbox{a.e. } \Omega\times Y^m  \\
& \mathrm{div} \int_{Y^m} \vec{u}_0(\vec{x},\vec{y}) \; \mathrm{d}\vec{y} = 0  \qquad \mbox{a.e. } \Omega
         \end{aligned}
\end{equation}
\end{proposition}

\begin{proposition}
\label{prop:eta_curl_free}
Let  $\{\vec{u}_\e\}$  be a curl free sequence,   uniformly bounded in\\ $L^2(\Omega;\R^3)$.
Then there exist a subsequence
 $\{\vec{u}_{\e_k}\}$
and a function $\vec{u}_0$ in\\  $L^2(\O,{\cal H}_{\sharp}(\mathrm{curl}_{{\bf R}_0}, Y^m)) $
such that
\begin{equation*}
\vec{u}_{\e_k} \wdto \vec{u}_0(\vec{x},\vec{y})
\end{equation*}
which satisfies
\begin{equation*}
        \begin{aligned}
& \mathrm{curl}_{\bf R}\, \vec{u}_0(\vec{x},\vec{y}) = \vec{0} \qquad \mbox{a.e. } \Omega\times Y^m  \\
&  \mathrm{curl}\int_{Y^m} \vec{u}_0(\vec{x},\vec{y}) \; \mathrm{d}\vec{y} = \vec{0}  \qquad \mbox{a.e. } \Omega
        \end{aligned}
\end{equation*}
\end{proposition}

\section{Illustrative examples in physics}\label{sec:Examples}

Let us now apply the two-scale cut-and-projection convergence to the homogenization of three problems of interest to the physics community.
\subsection{Homogenization of the electrostatic case}

Consider the quasiperiodic heterogeneous electrostatic problem, \ie
 \begin{equation}  \label{eq:static_eps}
     \left\{
        \begin{aligned}
          -  \mathrm{div} \; \sigma\left(\frac{{\bf R} \vec{x}}{\e}\right)  \mathrm{grad} u_\e(\vec{x}) & = f (\vec{x}) \;, \qquad \vec{x} \in \Omega\\
            u_\e|_{\partial\Omega} & = 0
        \end{aligned}
    \right.
 \end{equation}
where $f\in W^{-1,2}(\Omega)$ and $\sigma$ satisfies \eqref{eq:coersive}.
The solutions are uniformly bounded in $W^{1,2}_0(\Omega)$ with respect to $\e$.

\begin{theorem}\label{thm:electrostatic_hom}
The sequence of solutions  $\{u_\e\}$ that converges weakly in $W^{1,2}_0(\Omega)$  to the solution $\{u\}$ of the homogenized equation
 \begin{equation}  \label{eq:homogenized_static_strong}
     \left\{
        \begin{aligned}
   -  \mathrm{div} \;  \sigma^h  \nabla  u(\vec{x})  & =    f (\vec{x})  \;, \qquad \vec{x} \in \Omega\\
            u|_{\partial\Omega} & = 0
        \end{aligned}
        \right.
\end{equation}
where
\begin{equation}  \label{eq:homogenized_sigma_static}
     \sigma^h_{ik} =   \int_{Y^m} \sigma_{ij}\left(\vec{y} \right) \left( \delta_{jk} -  \left({\bf R}^T_j\nabla_{\vec{y}}\right)\chi^k (\vec{y})\right) \mathrm{d}\vec{y}
\end{equation}
and $\left({\bf R}^T_j\nabla_{\vec{y}}\right)\chi^k \in L^2(Y^m; \R^3)$
solves the local equation
\begin{equation}\label{eq:local_static}
    \int_{Y^m} \sigma_{ij}\left(\vec{y} \right) \left( \delta_{jk} -  \left({\bf R}^T_j\nabla_{\vec{y}}\right)\chi^k (\vec{y})\right)  \left({\bf R}^T_i\nabla_{\vec{y}}\right) \phi(\vec{y}) \; \mathrm{d}\vec{y}=0 \;, \qquad \phi \in W^{1,2}_\sharp(Y^m)
\end{equation}
supplied with periodic boundary conditions.
\end{theorem}
\begin{proof}
Choosing test functions $\e\,\phi$, where  $\phi \in D(\Omega;C^\infty_\sharp(Y^m))$ and defining \\
$\phi_\e(x):= \phi(x,\frac{{\bf R} \vec{x}}{\e})$ gives after an integration by parts
 \begin{equation*}  \label{eq:weak_static_eps}
    \int_\Omega \sigma\left(\frac{{\bf R} \vec{x}}{\e}\right)  \mathrm{grad} u_\e(\vec{x}) \cdot \left(\e \nabla \phi_\e(\vec{x}) + \nabla_{\bf R} \phi_\e(\vec{x})\right) \; \mathrm{d}\vec{x} = \int_\Omega f (\vec{x}) \e  \phi_\e(\vec{x})\; \mathrm{d}\vec{x}
\end{equation*}
Sending $\e \to 0$ gives due to Proposition~\ref{prop:grad-split}
 \begin{equation*}  \label{eq:locallimit_static}
    \int_\Omega  \int_{Y^m} \sigma\left(\vec{y} \right) \left(\nabla u(\vec{x}) +  \nabla_{\bf R} u_1(\vec{x}, \vec{y})\right) \cdot \nabla_{\bf R} \phi(\vec{x}, \vec{y}) \; \mathrm{d}\vec{x} \mathrm{d}\vec{y}= 0
\end{equation*}
Separation of variables by letting $ \nabla_{\bf R} u_1(\vec{x}, \vec{y}) = - \nabla_{\bf R} \chi^k (\vec{y})\partial_{\vec{x}_k}u(\vec{x})$ yields the local equation
\begin{equation*}  \label{eq:local_static_proof}
    \int_{Y^m} \sigma_{ij}\left(\vec{y} \right) \left( \delta_{jk} -  \left({\bf R}^T_j\nabla_{\vec{y}}\right)\chi^k (\vec{y})\right)  \left({\bf R}^T_i\nabla_{\vec{y}}\right) \phi(\vec{y}) \; \mathrm{d}\vec{y}=0 \;, \qquad \phi \in W^{1,2}_\sharp(Y^m)
\end{equation*}
The homogenized equation is obtained by choosing test function $\phi \in D(\Omega)$. We get analogously
 \begin{equation*}  \label{eq:globallimit_static}
    \int_\Omega  \int_{Y^m} \sigma\left(\vec{y} \right) \left( \nabla u(\vec{x}) +  \nabla_{\bf R} u_1(\vec{x}, \vec{y})\right) \cdot \nabla \phi(\vec{x}) \; \mathrm{d}\vec{x} \mathrm{d}\vec{y}=   \int_\Omega  \int_{Y^m} f (\vec{x})  \phi(\vec{x}) \; \mathrm{d}\vec{x} \mathrm{d}\vec{y}
\end{equation*}
That is
 \begin{equation*}  \label{eq:homogenized_static}
    \int_\Omega  \sigma^h  \nabla  u(\vec{x})  \cdot \nabla \phi(\vec{x}) \; \mathrm{d}\vec{x}  =   \int_\Omega   f (\vec{x})  \phi(\vec{x}) \; \mathrm{d}\vec{x}
\end{equation*}
where
\begin{equation*} 
          \sigma^h_{ik} =   \int_{Y^m} \sigma_{ij}\left(\vec{y} \right) \left( \delta_{jk} -  \left({\bf R}^T_j\nabla_{\vec{y}}\right)\chi^k (\vec{y})\right) \mathrm{d}\vec{y}
\end{equation*}
which completes the proof.
\end{proof}
The local equation \eqref{eq:local_static} provides a bounded gradient $\left({\bf R}_j^T\nabla_{\vec{y}}\right)\chi^k=\left({\bf R}^T\nabla_{\vec{y}}\right)_j\chi^k $ in $L^2(Y^m;\R^3) $, which follows by standard arguments.
\begin{proposition}
The local equation \eqref{eq:local_static} has a unique solution
$\nabla_{\bf R}\chi^k$  in\\ $L^2(Y^m;\R^3)$.
\end{proposition}
\begin{proof}
Let the   $\phi (\vec{y} )=  \chi^k$ in \eqref{eq:local_static}.
We get
\begin{equation*}
    \int_{Y^m} \sigma_{ij}\left(\vec{y} \right) \left( \delta_{jk} -  \left({\bf R}^T_j\nabla_{\vec{y}}\right)\chi^k (\vec{y})\right) \left({\bf R}^T_i\nabla_{\vec{y}}\right)\chi^k(\vec{y}) \; \mathrm{d}\vec{y}= 0
\end{equation*}
which by the  assumptions about coercivity  \eqref{eq:coersive} and that $\sigma$ is bounded certifies that there are two positive constants such that
\begin{equation*}
    c_1 \|{\nabla_{\bf R}}\chi^k \|^2_{L^2(Y^m;\R^3)} \leq c_2 \|{\nabla_{\bf R}}\chi^k \|_{L^2(Y^m;\R^3)}
\end{equation*}
\ie
\begin{equation*}
    \|{\nabla_{\bf R}}\chi^k \|_{L^2(Y^m;\R^3)} \leq \frac{c_2}{c_1}
\end{equation*}
Lax-Milgram theorem yields the result.
\end{proof}

Note that the higher dimensional elliptic equation has a degenerated elliptic kernel ${\bf R} \sigma {\bf R}^T$, by the same reason as in Lemma~\ref{lem:R-poisson_eqn},  which implies that we do not necessarily have a bounded potential  $\chi^k \in W^{1,2}(Y^m)$. But, as in Lemma~\ref{lem:R-poisson_eqn}, by imposing the constraint $\left({\bf I}_m -  {\bf R} \sigma {\bf R}^T\right)\nabla_{\vec{y}} \chi^k =  \vec{0}$ on the solutions we will certify bounded solutions.
 Doing this, restricts the currents to the $n$-dimensional hyperplane in $\R^m$. This is \eg reflected by \eqref{eq:local_static} which we solve only for the projected gradient on this hyperplane. This gradient can then be  used to get the potential on this hyperplane by integrating the projected gradients in the plane.
Another, and from a numerical implementation point of view  interesting alternative is to change variables in $\R^m$, \ie by rotating the coordinate system  to make the hyperplane parallel with the new (real) coordinate axes. This is done by finding the eigenvalues of ${\bf R}{\bf R}^T$. Obviously, the columns of ${\bf R}$ are all eigenvectors,  corresponding to the eigenvalue $\lambda=1$, of multiplicity $n$. This follows since ${\bf R}^T{\bf R} = {\bf I}_n$. The other ($m-n$) eigenvectors are in the kernel of  ${\bf R}{\bf R}^T$, \ie they correspond to the eigenvalue $\lambda=0$. We  denote the subspace spanned by them as ${\bf R}^{\perp}$. Defining new coordinates in $\R^m$ as $\vec{y}' = \left[ {\bf R} \; {\bf R}^{\perp}   \right]\vec{y}$ will give us the conductivity tensor $\sigma$ on the diagonal, and all other entries will be zero.
Since there is no cross talk between the first $n$ components of the gradient $\nabla_{\vec{y}'} \chi^k$, we can  truncate the vector and solve the higher dimensional system in $n$ dimensions. The higher dimension only comes into account in the description of the conductivity as periodic in $\R^m$. One could of course also regularize the higher dimensional equation by adding a small conductivity in the direction of the degeneracy in $\sigma$. In the rotated coordinate system that would correspond to introducing $\alpha I_m$ on the lower part of the diagonal, where $\alpha >0$ is the regularizing parameter.

We have the following, not optimal, corrector result. It can be made stronger using the method in \cite{Cherednichenko+Cooper2016}
\begin{proposition}[Correctors]\label{prop:correctors}
Let $u_\e$ and $u$ be  solutions of \eqref{eq:static_eps} and \eqref{eq:homogenized_static_strong}, respectively, and let $\nabla_{\bf R} \vec{\chi}$ solve \eqref{eq:local_static}, then
\begin{equation*}
 \lim_{\eta \to 0} \left\|  \nabla  u_\e(\vec{x}) - \nabla  u(\vec{x}) - \nabla  u(\vec{x})  \cdot \nabla_{\bf R} \vec{\chi} \left(\frac{{\bf R} \vec{x}}{\e}\right)  \right\|_{L^2(Y^m;\R^3)} = 0
\end{equation*}
\end{proposition}
Here, $\vec{\chi}$ is a vector with components $\chi^k$.
\begin{proof}
The proof follows the corresponding proof in \cite{Allaire1992}.
The coercivity assumption on the material property  yields
\begin{equation*}
  \begin{aligned}
     &c_1\left\|  \nabla  u_\e(\vec{x}) - \nabla  u(\vec{x}) - \nabla  u(\vec{x})  \cdot \nabla_{\bf R}\vec{\chi}   \left(\frac{{\bf R} \vec{x}}{\e}\right) \right\|^2  \leq \\
    & \int_{\Omega} \sigma\left(\frac{{\bf R} \vec{x}}{\e}\right)  \left[\nabla  u_\e(\vec{x}) - \nabla  u(\vec{x}) - \nabla  u(\vec{x})  \cdot \nabla_{\bf R}\vec{\chi}   \left(\frac{{\bf R} \vec{x}}{\e}\right) \right] \cdot \\
    & \left[\nabla  u_\e(\vec{x}) - \nabla  u(\vec{x}) - \nabla  u(\vec{x})  \cdot \nabla_{\bf R}\vec{\chi}   \left(\frac{{\bf R} \vec{x}}{\e}\right) \right]
    \; \mathrm{d}\vec{x} = \\
    & \underset{I_1}{\underbrace{\int_{\Omega}   f(\vec{x}) u_\e(\vec{x}) \; \mathrm{d}\vec{x} }}  -
    \underset{I_2}{\underbrace{\int_{\Omega} \sigma\left(\frac{{\bf R} \vec{x}}{\e}\right)  \nabla  u_\e(\vec{x})\cdot \left[ \nabla  u(\vec{x}) + \nabla  u(\vec{x})  \cdot \nabla_{\bf R}\vec{\chi}   \left(\frac{{\bf R} \vec{x}}{\e}\right)
    \right]  \; \mathrm{d}\vec{x}}} - \\
    &  \underset{I_3}{\underbrace{  \int_{\Omega} \sigma\left(\frac{{\bf R} \vec{x}}{\e}\right) \left[ \nabla  u(\vec{x}) + \nabla  u(\vec{x})  \cdot \nabla_{\bf R}\vec{\chi}   \left(\frac{{\bf R} \vec{x}}{\e}\right)
    \right]  \cdot \nabla  u_\e(\vec{x}) \; \mathrm{d}\vec{x}} } + \\
    &\underset{I_4}{\underbrace{ \int_{\Omega} \sigma\left(\frac{{\bf R} \vec{x}}{\e}\right) \left[ \nabla  u(\vec{x}) + \nabla  u(\vec{x})  \cdot \nabla_{\bf R}\vec{\chi}   \left(\frac{{\bf R} \vec{x}}{\e}\right)
    \right] \cdot \left[ \nabla  u(\vec{x}) + \nabla  u(\vec{x})  \cdot \nabla_{\bf R}\vec{\chi}   \left(\frac{{\bf R} \vec{x}}{\e}\right)
    \right]\; \mathrm{d}\vec{x}} }
  \end{aligned}
\end{equation*}
The integrals are evaluated individually, using the assumption that the local solutions are admissible test functions:
\begin{equation*}
    I_1 \to \int_{\Omega}   f(\vec{x}) u (\vec{x}) \; \mathrm{d}\vec{x}
\end{equation*}
\begin{equation*}
    I_2 \to \int_{\Omega}\int_{Y^m} \sigma\left(\vec{y} \right)  \left(  \nabla u(\vec{x}) +  \nabla_{\bf R} u_1(\vec{x}, \vec{y})\right)
     \cdot \left[ \nabla  u(\vec{x}) + \nabla  u(\vec{x})  \cdot \nabla_{\bf R}\vec{\chi} ( \vec{y} )
    \right]  \; \mathrm{d}\vec{y} \mathrm{d}\vec{x}
\end{equation*}
similarly
\begin{equation*}
    I_3 \to \int_{\Omega}\int_{Y^m} \sigma\left(\vec{y} \right) \left[ \nabla  u(\vec{x}) + \nabla  u(\vec{x})  \cdot \nabla_{\bf R}\vec{\chi} ( \vec{y} )
    \right] \cdot
    \left( \nabla u(\vec{x}) +  \nabla_{\bf R} u_1(\vec{x}, \vec{y})\right)
       \; \mathrm{d}\vec{y} \mathrm{d}\vec{x}
\end{equation*}
and
\begin{equation*}
    I_4 \to \int_{\Omega}\int_{Y^m} \sigma\left(\vec{y} \right) \left[ \nabla  u(\vec{x}) + \nabla  u(\vec{x})  \cdot \nabla_{\bf R}\vec{\chi} ( \vec{y} )
    \right] \cdot
    \left[ \nabla  u(\vec{x}) + \nabla  u(\vec{x})  \cdot \nabla_{\bf R}\vec{\chi} ( \vec{y} )
    \right]
       \; \mathrm{d}\vec{y} \mathrm{d}\vec{x}
\end{equation*}
We find that the limits of $-I_3+I_4 = 0$. Further, we get
\begin{equation*}
    \int_{\Omega}   f(\vec{x}) u (\vec{x}) \; \mathrm{d}\vec{x}  -
    \int_{\Omega}\int_{Y^m} \sigma\left(\vec{y} \right)  \left( \nabla u(\vec{x}) +  \nabla_{\bf R} u_1(\vec{x}, \vec{y})\right)
     \cdot   \nabla  u(\vec{x})   \; \mathrm{d}\vec{y} \mathrm{d}\vec{x} =0
\end{equation*}
due to the homogenized equation and
\begin{equation*}
    \int_{\Omega}\int_{Y^m} \sigma\left(\vec{y} \right)  \left( \nabla u(\vec{x}) +  \nabla_{\bf R} u_1(\vec{x}, \vec{y})\right)
     \cdot    \nabla  u(\vec{x})  \cdot \nabla_{\bf R}\vec{\chi} ( \vec{y} ) \; \mathrm{d}\vec{y} \mathrm{d}\vec{x} =0
\end{equation*}
due to the local equations.
\end{proof}

\subsection{Homogenization of the elastostatic case}

Consider the elasticity problem
 \begin{equation}  \label{eq:elastostatic_eps}
     \left\{
        \begin{aligned}
          -  \frac{\partial}{\partial x_j} \; C_{ijkl}\left(\frac{{\bf R} \vec{x}}{\e}\right)  \frac{\partial}{\partial x_l} {u_k}_\e(\vec{x}) & = f_i (\vec{x}) \;, \qquad  \vec{x} \in  \Omega \\
            {\bf u}_\e|_{\partial\Omega} & = {\bf 0}
        \end{aligned}
    \right.
 \end{equation}
where ${\bf f} \in W^{-1,2}(\Omega,\R^3)$
and $C_{ijkl}$ is the symmetric elasticity tensor that is bounded and coercive satisfying \eqref{eq:coersive_2}. 
The solutions are uniformly bounded in $W^{1,2}(\Omega,\R^3)$ with respect to $\e$.
\begin{theorem}\label{thm:elasto_hom}
There exists a subsequence of $\{{\bf u}_\e\}$ that converges weakly in\\
$W^{1,2}_0(\Omega)$  to the solution $\{{\bf u}\}$ of the homogenized equation
 \begin{equation}  \label{eq:homogenized_static_elast}
     \left\{
        \begin{aligned}
   - \frac{\partial}{\partial x_j} \; C_{ijkl}^h \;  \frac{\partial}{\partial x_l}  u_k(\vec{x})  & =    f_i (\vec{x})  \;, \qquad  \vec{x} \in  \Omega \\
            {\bf u}|_{\partial\Omega} & = {\bf 0}
        \end{aligned}
        \right.
\end{equation}
where
\begin{equation}  \label{eq:homogenized_sigma_static_elast}
     C_{ijkl}^h =   \int_{Y^m} C_{ijpq}\left(\vec{y} \right) \left( \delta_{pk}\delta_{ql} -  \nabla_{{\bf R}_q}\chi_p^{kl} (\vec{y})\right) \mathrm{d}\vec{y}
\end{equation}
and $\nabla_{\bf R_q}\chi_p^{kl} \in L^2(Y^m;\R^3)$
\begin{equation}\label{eq:local_elastostatic}
    \int_{Y^m} C_{ijpq}\left(\vec{y} \right) \left( \delta_{pk}\delta_{ql} -  \nabla_{{\bf R}_q}\chi_p^{kl} (\vec{y})\right) \cdot {\nabla_{\bf R}}_j \phi_i(\vec{y}) \; \mathrm{d}\vec{y}= 0
\end{equation}
with periodic boundary conditions.
\end{theorem}
\begin{proof}
It follows mutatis mutandis the proof of Theorem \ref{thm:electrostatic_hom}.
\end{proof}

\begin{proposition}
The local equation \eqref{eq:local_elastostatic} has a unique solution

$
\nabla_{\bf R}\chi_p^{kl} \in L^2(Y^m;\R^9)
$
\end{proposition}
\begin{proof}
Let the test function $\phi_i(\vec{y}) = \chi_i^{kl}(\vec{y})$ in \eqref{eq:local_elastostatic}.
We get
\begin{equation*}
    \int_{Y^m} C_{ijpq}\left(\vec{y} \right) \left( \delta_{pk}\delta_{ql} -  \nabla_{{\bf R}_q}\chi_p^{kl} (\vec{y})\right) \cdot {\nabla_{\bf R}}_j \chi_{i}^{kl}(\vec{y}) \; \mathrm{d}\vec{y}= 0
\end{equation*}
which gives
\begin{equation*}
    c_1 \|{\nabla_{\bf R}}\chi^{kl} \|^2_{L^2(Y^m;\R^9)} \leq c_2 \|{\nabla_{\bf R}}\chi^{kl} \|_{L^2(Y^m;\R^9)}
\end{equation*}
\ie
\begin{equation*}
    \|{\nabla_{\bf R}}\chi^{kl} \|_{L^2(Y^m;\R^9)} \leq \frac{c_2}{c_1}
\end{equation*}
\end{proof}
Note that we do not have a bound for $\chi_i^{kl}\in L^2(Y^m)$.

\subsection{Homogenization of the quasistatic magnetic case}
Consider the  quasi\-periodic heterogeneous quasistatic magnetic problem
 \begin{equation}  \label{eq:magnetostatic_eps}
     \left\{
        \begin{aligned}
          \mathrm{curl} \, \epsilon^{-1}\left(\frac{{\bf R} \vec{x}}{\e}\right)  \mathrm{curl} \, \vec{u}_\e(\vec{x}) & = \vec{f} (\vec{x}) \,, \qquad \vec{x} \in \Omega\\
            \hat{\vec{\nu}}\times\epsilon^{-1}\mathrm{curl}\,\vec{u}_\e|_{\partial\Omega} & = \vec{0}
        \end{aligned}
    \right.
 \end{equation}
where $ \hat{\vec{\nu}}$ is the unit normal to the boundary, $ \epsilon^{-1}$ is bounded and corecive \eqref{eq:coersive_3} and  the driving term $\vec{f}$ belongs to the dual of ${\cal H}_0(\mathrm{curl},\Omega)$. The solutions are uniformly bounded in ${\cal H}_0(\mathrm{curl},\Omega)$ with respect to $\e$.

\begin{theorem}\label{thm:quasistatic_magn_homog}
There exists a subsequence of $\{\vec{u}_\e\}$ that converges weakly in \\ ${\cal H}_0(\mathrm{curl},\Omega)$  to the solution $\{\vec{u}\}$ of the homogenized equation
 \begin{equation}  \label{eq:homogenized_magnetostatic_strong}
     \left\{
        \begin{aligned}
   \mathrm{curl} \,  \epsilon^{-1}_h  \mathrm{curl} \, \vec{u}(\vec{x})  & =    \vec{f} (\vec{x})  \;, \qquad \vec{x} \in \Omega\\
               \hat{\vec{\nu}}\times\epsilon^{-1}_h\mathrm{curl}\,\vec{u}|_{\partial\Omega} & = \vec{0}
        \end{aligned}
        \right.
\end{equation}
where
\begin{equation}  \label{eq:homogenized_epsilon_static}
     \epsilon^{-1}_{ij,h} =   \int_{Y^m} \epsilon^{-1}_{ij}\left(\vec{y} \right) \left( \delta_{jk} - \left(\left({\bf R}^T\nabla_{\vec{y}}\right)\times {\vec{\chi}^k} (\vec{y})\right)\right)_j\; \mathrm{d}\vec{y}
\end{equation}
and $\left({\bf R}^T\nabla_{\vec{y}}\right)\times\vec{\chi}^k$ solves the local equation
\begin{equation}\label{eq:local_magnetostatic_1}
    \int_{Y^m} \epsilon^{-1}_{ij}\left(\vec{y} \right) \left( \delta_{jk} -  \left(\left({\bf R}^T\nabla_{\vec{y}}\right)\times\vec{\chi}^k (\vec{y})\right)\right)_j   \left(\left({\bf R}^T\nabla_{\vec{y}}\right)\times\vec{\phi}(\vec{y})\right)_i  \; \mathrm{d}\vec{y}=0
\end{equation}
with periodic boundary conditions and $\vec{\phi}\in
{\cal H}_{\sharp}(\mathrm{curl}_{\bf R}, Y^m)$.
\end{theorem}
\begin{proof}
Choosing test functions $\e\, \vec{\phi}$, where  $\vec{\phi} \in D(\Omega;C^\infty_\sharp(Y^m,\R^3))$ and defining $\vec{\phi}_\e(\vec{x}):= \vec{\phi}(\vec{x},\frac{{\bf R} \vec{x}}{\e})$ gives after an integration by parts
 (or choose $\vec{\phi}_1\in{\cal H}(\mathrm{curl},\Omega)$, $\vec{\phi}_2\in
{\cal H}_{\sharp}(\mathrm{curl}_{\bf R}, Y^m)$ and use e.g. $\vec{\phi} = \vec{\phi}_1\otimes \vec{\phi}_2$ or  $\vec{\phi}_1 \in D(\Omega)$ and define  $\vec{\phi} = \vec{\phi}_1 \vec{\phi}_2$).
\begin{equation*}  \label{eq:weak_magnetostatic_eps}
       \begin{aligned}
&\int_\Omega \left(\epsilon^{-1}\left(\frac{{\bf R} \vec{x}}{\e}\right)  \mathrm{curl} \, \vec{u}_\e(\vec{x}) \right) \cdot \left(\e \,
\mathrm{curl} \,\vec{\phi}_\e(\vec{x}) + \mathrm{curl}_{\bf R} \vec{\phi}_\e(\vec{x})\right) \; \mathrm{d}\vec{x} \\
&= \int_\Omega \vec{f} (\vec{x}) \cdot \e \, \vec{\phi}_\e(\vec{x})\; \mathrm{d}\vec{x}
       \end{aligned}
\end{equation*}
Sending $\e \to 0$ gives due to Proposition~\ref{prop:curl_split}
 \begin{equation*}  \label{eq:locallimit_magnetostatic}
    \int_\Omega  \int_{Y^m} \left(\epsilon^{-1}\left(\vec{y} \right) \left( \mathrm{curl}\, \vec{u}(\vec{x}) +  \mathrm{curl}_{\bf R} \,\vec{u}_1(\vec{x}, \vec{y})\right)\right) \cdot \mathrm{curl}_{\bf R}\, \vec{\phi}(\vec{x}, \vec{y}) \; \mathrm{d}\vec{x} \mathrm{d}\vec{y}= 0
\end{equation*}
 We separate the variables by assuming
$ \mathrm{curl}_{\bf R} \,\vec{u}_1(\vec{x}, \vec{y}) =-(\nabla_{\bf R} \times \vec{\chi}^k (\vec{y})) \,(\nabla_x\times\vec{u}(\vec{x}))_k$
yields the local equation
\begin{equation*} 
    \int_{Y^m} \left( \epsilon^{-1}_{ij}\left(\vec{y} \right) \left( \delta_{jk} -  \left(\nabla_{\bf R} \times \vec{\chi}^k (\vec{y})\right)_j\right) \right) \left(\nabla_{\bf R} \times \vec{\phi}_2(\vec{y})\right)_i  \; \mathrm{d}\vec{y}=0  \; , \; \vec{\phi}_2\in
{\cal H}_{\sharp}(\mathrm{curl}_{\bf R}, Y^m)
\end{equation*}
This equation has bounded solution $\nabla_{\bf R} \times \vec{\chi}^k (\vec{y})$ in $L^2(Y^m;\R^3)$, which follows by coercivity and Lax-Milgrams Lemma.
The homogenized equations are obtained by choosing test function $\vec{\phi}_1\in{\cal H}(\mathrm{curl},\Omega)$. We get in the limit
 \begin{equation*}  
      \begin{aligned}
&\int_\Omega  \int_{Y^m} \epsilon^{-1}\left(\vec{y} \right) \left( \mathrm{curl} \vec{u}(\vec{x}) +  \mathrm{curl}_{\bf R} \vec{u}_1(\vec{x}, \vec{y})\right) \cdot \mathrm{curl} \phi(\vec{x}) \; \mathrm{d}\vec{x} \mathrm{d}\vec{y} \\
&=   \int_\Omega  \int_{Y^m} \vec{f} (\vec{x}) \cdot \phi(\vec{x}) \; \mathrm{d}\vec{x} \mathrm{d}\vec{y}
      \end{aligned}
\end{equation*}

That is
 \begin{equation*}  
    \int_\Omega  \epsilon^{-1}_h  \mathrm{curl}  \vec{u}(\vec{x})  \cdot \mathrm{curl} \phi(\vec{x}) \; \mathrm{d}\vec{x}  =   \int_\Omega   \vec{f} (\vec{x}) \cdot \phi(\vec{x}) \; \mathrm{d}\vec{x}
\end{equation*}
where the homogenized coefficient is given by
\begin{equation*} 
          \epsilon^{-1}_{ik,h} =   \int_{Y^m} \epsilon^{-1}_{ij}\left(\vec{y} \right) \left( \delta_{jk} -  \left(\nabla_{\bf R} \times \vec{\chi}^k (\vec{y})\right)_j \right) \; \mathrm{d}\vec{y}
\end{equation*}
which concludes the proof.
\end{proof}

\begin{remark}
Regarding uniqueness of the solution to the local equation, it is just like in Theorem~\ref{thm:electrostatic_hom}, up to replacement of gradients by curls. Likewise for the corrector, it is similar to Proposition~\ref{prop:correctors} up to replacement of gradients by curls. We further note that the main result of \cite{Bouchitte+etal2010}, Theorem 1.3, makes use of a continuous linear map to deduce two-scale convergence of curls from that of gradients (in Fourier space) for the homogenization of the Maxwell system.
\end{remark}

Let us now study some illustrative examples.

\subsection{Illustrative examples: quasiperiodic two-dimensional and layered media}
We will give two electrostatic examples:  one dimensional quasiperiodic composites, modelling a laminate, and a two dimensional case defined by the Penrose tiling.
\subsubsection{Homogenization of a quasiperiodic layered medium}
Let us start with a one-dimensional quasicrystal which can be homogenized analytically. An example of a periodic medium whose cut-and-projection generates a so-called Fibonacci quasicrystal which is shown in figure \ref{figsiam}. This is a straightforward application of Theorem~\ref{thm:electrostatic_hom}, where we consider a cut-and-projection with $n=1$, $m=2$ and  ${\bf R}^T=(1,\tau)$.
The auxiliary problem \eqref{eq:local_static} takes the form
\begin{equation*}
-\sum_{i=1}^{2}R_i\frac{\partial}{\partial y_i}\left(\sigma(\vec{y})R_i\frac{\partial}{\partial y_i}\chi(\vec{y})\right)
=\left(\sum_{i=1}^{2}R_i\frac{\partial}{\partial y_i}\sigma(\vec{y})\right) \; ,
\end{equation*}
with $\chi(\vec{y})$ $Y^2$-periodic.
Integrating over $Y^2$, we deduce that
\begin{equation}\label{eq:part_int_Ym}
-\sigma(\vec{y})\sum_{i=1}^{2}R_i\frac{\partial}{\partial y_i}\chi(\vec{y})=\sigma(\vec{y})+C \; ,
\end{equation}
where $C$ is an integration constant, so that
\begin{equation*}
-\sum_{i=1}^{2}R_i\frac{\partial}{\partial y_i}\chi(\vec{y})=1+C\sigma^{-1}(\vec{y}) \; .
\end{equation*}
From the $Y^2$ periodicity of $\chi(\vec{y})$, we conclude that
\begin{equation}
\label{eq:1dconst}
\int_{Y^2} \left( 1+C\sigma^{-1}(\vec{y}) \right) d\vec{y} = 0 \; .
\end{equation}
From \eqref{eq:homogenized_sigma_static}, we finally obtain the expression for the homogenized coefficient
\begin{equation*}
\begin{aligned}
\sigma^h &=\displaystyle{\int_{Y^2}\left(
\sigma(\vec{y})+\sigma(\vec{y})\sum_{i=1}^{2}R_i\frac{\partial}{\partial y_i}\chi(\vec{y})\right) \, d\vec{y}} \\
&= \displaystyle{\int_{Y^2}\sigma(\vec{y})d\vec{y} -\int_{Y^2}\left(\sigma(\vec{y})+C\right)d\vec{y}}
= \displaystyle{{\left(\int_{Y^2}\sigma^{-1}(\vec{y})d\vec{y}\right)}^{-1}} \; ,
\end{aligned}
\label{effective1d}
\end{equation*}
where we have used \eqref{eq:part_int_Ym} and (\ref{eq:1dconst}). One notes that specific values of $\tau$ in {\bf R} is not needed in the computation of the homogenized coefficient, which is simply a harmonic mean in $Y^2$.

\subsubsection{Homogenization of a Penrose tiling}
We would like finally to study a second type of quasicrystal  Al-Mn alloy, which is the decagonal phase discovered by Bendersky in 1985 \cite{Bendersky1985}. This phase is periodic along one direction on the diffraction diagram, and quasiperiodic in the other two directions. This quasicrystal can be viewed as Penrose tilings stacked on top of one another along the direction of periodicity. We note that the homogenization of Penrose tilings has been solved using an energy approach in \cite{Braides2009}.

We shall focus here on the analysis of effective properties  of the Al-Mn alloy in the transverse quasiperiodic plane, and thus disregard the third, periodic, direction.
Mathematically, a Penrose tiling can be generated from a four-dimensional periodic structure.
The matrix transpose of ${\bf R}$ describing a Penrose tiling \cite{Jaric1988}  is given by
\begin{equation*}
{\bf R}^T = \frac{1}{\sqrt{2}}
\begin{pmatrix}
\frac{\tau-1}{\sqrt{3}} & \frac{-\tau}{\sqrt{3}} & \frac{-\tau}{\sqrt{3}} & \frac{\tau-1}{\sqrt{3}} \\ \\
\frac{\sqrt{\tau+2}}{\sqrt{5}} & \frac{\sqrt{3-\tau}}{\sqrt{5}} & \frac{-\sqrt{3-\tau}}{\sqrt{5}} & \frac{-\sqrt{\tau+2}}{\sqrt{5}}
\end{pmatrix}
\end{equation*}
Making use of Theorem \ref{thm:electrostatic_hom}, the effective conductivity of the Penrose tiling
is given by
\begin{equation*}
\sigma^h = {\int_{Y^4}\sigma\left(\vec{y} \right)}
\left [
 {\bf I} - {\sum_{k=1}^4}\left(
  \begin{aligned}
  &
 {R_{k1}\frac{\partial \chi^1\left(\vec{y} \right)}{\partial y_k}}  \;\; {R_{k1}\frac{\partial \chi^2\left(\vec{y} \right)}{\partial y_k}}\\\\
 &
 {R_{k2}\frac{\partial \chi^1\left(\vec{y} \right)}{\partial
 y_k}} \; \; {R_{k2}\frac{\partial \chi^2\left(\vec{y} \right)}{\partial y_k}}
\end{aligned}
 \right )
\right ] \, \mathrm{d}\vec{y}
\end{equation*}
and $\left({\bf R}^T_j\nabla_{\vec{y}}\right)\chi^k \in L^2(Y^4; \R^2)$
solves the local equation
\begin{equation*}
 {\int_{Y^4}\sigma\left(\vec{y} \right)}
\left [
 {\bf I} - {\sum_{k=1}^4}\left(
  \begin{aligned}
  &
 {R_{k1}\frac{\partial \chi^1\left(\vec{y} \right)}{\partial y_k}}  \;\; {R_{k1}\frac{\partial \chi^2\left(\vec{y} \right)}{\partial y_k}}\\\\
 &
 {R_{k2}\frac{\partial \chi^1\left(\vec{y} \right)}{\partial
 y_k}}  \;\; {R_{k2}\frac{\partial \chi^2\left(\vec{y} \right)}{\partial y_k}}
\end{aligned}
 \right )
\right ]
 \sum_{l=1}^4
\left (
  \begin{aligned}
 &
 R_{l1}
 \frac{\partial \phi\left(\vec{y} \right)}
 {\partial y_l}\\
 &
  R_{l2}
 \frac{\partial \phi\left(\vec{y} \right)}
 {\partial y_l}
\end{aligned}
 \right)
\mathrm{d}\vec{y} =
\left (
  \begin{aligned}
 & 0 \\
 & 0
\end{aligned}
 \right )
\end{equation*}

\section{Concluding remarks}\label{sec:Conclusions}
In this paper,
we have derived a number of weak and strong compactness results for two-scale  cut-and-projection convergence. This concept allows for rapid identification of limits  of sequences of solutions and of the corresponding  PDEs with fast oscillating periodic coefficients in a higher dimensional space, which model quasiperiodic oscillations in a projected physical space. Amongst illustrative examples, we treated the homogenization of the electrostatic, elastostatic,  and quasistatic equations in quasiperiodic media.  Our results can be adapted to homogenization of quasiperiodic perforated and porous media. Importantly, most results stated here for sequences of functions in $L^2$ spaces, can be translated to $L^p$ spaces, including $L^1$ with a notion of two-scale cut-and-projection convergence of measures. Work is also in progress regarding the extension of our results to non-linear PDEs, reiterated homogenization of quasiperiodic multiscale media and homogenization of quasiperiodic spectral problems. For the latter, one could make use of Bloch's theorem in the higher dimensional space, and computations suggest some interesting self-similar features of band structure at long wavelengths \cite{Zolla1998,Rodriguez+etal2008}. Finally, the irrational operators introduced in Section 4 have been used to adapt the classical two-scale asymptotic approach to the quasiperiodic setting \cite{Cherkaev+etal2019}.

\section*{Acknowledgments}
SG wishes to thank the Department of Mathematics at Imperial College London for a visiting position in the group of Prof. R.V. Craster in 2018-2019 (funded by EPSRC grant EP/L024926/1). SG also acknowledges the Unit\'e Mixte Internationale CNRS-Imperial Abraham de Moivre and funding from department INSIS of CNRS. EC acknowledges support from the U.S. NSF through grant DMS-1715680.

\appendix
\section{Appendix}\label{sec:Appendix}
We have the following integration by parts lemmas

\begin{lemma}[Green's Identity] \label{Integration_by_parts_green}
It holds that
\begin{equation}
 -\int_{Y^m} \mathrm{div}_{\bf R} \; \vec{\phi}(\vec{y})  \;   {\theta}(\vec{y}) \;\mathrm{d}\vec{y} = \int_{Y^m}    \vec{\phi}(\vec{y})  \cdot  \mathrm{grad}_{\bf R} \;{\theta}(\vec{y}) \;\mathrm{d}\vec{y}
\end{equation}
for  $\vec{\phi} \in  H_{\sharp}(\mathrm{div}_{\bf R}, Y^m)$ and  $\theta \in  H_{\sharp}(\mathrm{grad}_{\bf R}, Y^m)$.
\end{lemma}
\begin{proof}
By the definition of operators $\mathrm{div}_{\bf R}$ and $\mathrm{grad}_{\bf R}$,  Remark~\ref{rem:divR_and_curlR} and periodic boundary conditions it follows that
\begin{equation}
\begin{aligned}
&
 -\int_{Y^m} \mathrm{div}_{\bf R} \; \vec{\phi}(\vec{y})  \;   {\theta}(\vec{y}) \;\mathrm{d}\vec{y} =
 -\int_{Y^m} ({\bf R}^T\nabla_y ) \cdot \vec{\phi}(\vec{y})  \;   {\theta}(\vec{y}) \;\mathrm{d}\vec{y} =
 \\
 &
 -\int_{Y^m} \nabla_y  \cdot {\bf R}\vec{\phi}(\vec{y})  \;   {\theta}(\vec{y}) \;\mathrm{d}\vec{y} =
 \int_{Y^m}   {\bf R} \vec{\phi} \cdot  \nabla_y \;{\theta}(\vec{y}) \;\mathrm{d}\vec{y}=
 \\
 &
 \int_{Y^m}    \vec{\phi}(\vec{y})  \cdot  {\bf R}^T\nabla_y \;{\theta}(\vec{y}) \;\mathrm{d}\vec{y}=
 \int_{Y^m}    \vec{\phi}(\vec{y})  \cdot  \mathrm{grad}_{\bf R}\; {\theta}(\vec{y}) \;\mathrm{d}\vec{y}
\end{aligned}
\end{equation}
for any pair of functions $\vec{\phi} \in  H_{\sharp}(\mathrm{div}_{\bf R}, Y^m)$ and  $\theta \in  H_{\sharp}(\mathrm{grad}_{\bf R}, Y^m)$.
\end{proof}

\begin{lemma}[Stokes' theorem] \label{Integration_by_parts_Stokes}
It holds that
\begin{equation}
 \int_{Y^m} \mathrm{curl}_{\bf R} \; \vec{\phi}(\vec{y})   \cdot  \vec{\theta}(\vec{y}) \;\mathrm{d}\vec{y} =
 \int_{Y^m}    \vec{\phi}(\vec{y})  \cdot  \mathrm{curl}_{\bf R} \; \vec{\theta}(\vec{y}) \;\mathrm{d}\vec{y}
\end{equation}
for all $\vec{\phi},  \vec{\theta} \in  H_{\sharp}(\mathrm{curl}_{\bf R}, Y^m)$.
\end{lemma}
\begin{proof}
Using the observation in Remark~\ref{rem:divR_and_curlR}, we can invoke the similar arguments as in Lemma~\ref{Integration_by_parts_green} to get
\begin{equation}
\begin{aligned}
&
\int_{Y^m} \mathrm{curl}_{\bf R} \; \vec{\phi}(\vec{y})   \cdot  \vec{\theta}(\vec{y}) \;\mathrm{d}\vec{y} =
 \int_{Y^m}    \vec{\phi}(\vec{y})  \cdot  \mathrm{curl}_{\bf R} \; \vec{\theta}(\vec{y}) \;\mathrm{d}\vec{y}
\end{aligned}
\end{equation}
for any pair of functions $\vec{\phi},\vec{\theta}\in  H_{\sharp}(\mathrm{curl}_{\bf R}, Y^m)$.
\end{proof}


\bibliographystyle{unsrt}
\bibliography{quasirefsarxiv}



\end{document}